\documentclass[review,onefignum,onetabnum]{siamart220329}

\DeclareMathOperator{\fl}{fl}
\DeclareMathOperator{\op}{op}
\newcommand{\card}{\mathrm{Card}}
\newcommand{\Cov}{\mathrm{Cov}}

\nolinenumbers

\usepackage{lipsum}
\usepackage{amsfonts}
\usepackage{graphicx}
\usepackage{epstopdf}
\usepackage{xcolor}
\usepackage{tikz}
\usepackage{listings}
\usepackage{caption}
\usepackage{float}
\usepackage{xspace}
\usepackage{amsmath,amssymb,amsfonts}
\usepackage{booktabs}
\usepackage{array,multirow,makecell}
\usepackage{stmaryrd}
\usepackage{scalerel}
\usepackage{csquotes}
\usepackage{subfig}
\usepackage[symbol]{footmisc}

\usepackage{algorithmic}
\ifpdf
  \DeclareGraphicsExtensions{.eps,.pdf,.png,.jpg}
\else
  \DeclareGraphicsExtensions{.eps}
\fi

\newsiamremark{remark}{Remark}
\newsiamremark{hypothesis}{Hypothesis}
\newsiamremark{pro}{Proposition}
\newsiamremark{mode}{Model}
\newsiamremark{defn}{Definition}
\crefname{hypothesis}{Hypothesis}{Hypotheses}
\newsiamthm{claim}{Claim}

\headers{Stochastic Rounding}{E. M. EL ARAR, D. SOHIER, P. DE OLIVEIRA CASTRO, and E. PETIT}

\title{Stochastic Rounding Variance and Probabilistic Bounds: A New Approach\thanks{Submitted to the editors 21 jul 2022.}
\funding{This research was supported by the French National Agency for Research (ANR) via the InterFLOP project (No. ANR-20-CE46-0009).}}

\author{El-Mehdi El Arar,\thanks{Université Paris-Saclay, UVSQ, Li-PaRAD, Saint-Quentin en Yvelines, France
  (\email{el-mehdi.el-arar@uvsq.fr},\email{ devan.sohier@uvsq.fr}, \email{pablo.oliveira@uvsq.fr}).}
  \and Devan Sohier\footnotemark[2]  \and Pablo de Oliveira Castro\footnotemark[2]
\and Eric Petit\thanks{Intel Corp
  (\email{eric.petit@intel.com}).}}

\usepackage{amsopn}

\newcolumntype{R}[1]{>{\raggedleft\arraybackslash }b{#1}}
\newcolumntype{L}[1]{>{\raggedright\arraybackslash }b{#1}}
\newcolumntype{C}[1]{>{\centering\arraybackslash }b{#1}}

\ifpdf
\hypersetup{
  pdftitle={Stochastic Rounding},
  pdfauthor={El-Mehdi El Arar, Devan Sohier, Pablo de Oliveira Castro, and Eric Petit}
}
\fi

\begin{document}

\maketitle

\begin{abstract}

Stochastic rounding (SR) offers an alternative to the deterministic IEEE-754 floating-point rounding modes. In some applications such as PDEs, ODEs, and neural networks, SR empirically improves the numerical behavior and convergence to accurate solutions while the theoretical background remains partial.
Recent works by Ipsen, Zhou, Higham, and Mary have computed SR probabilistic error bounds for basic linear algebra kernels. For example, the inner product SR probabilistic bound of the forward error is proportional to $\sqrt{n}u$ instead of $nu$ for the default rounding mode. To compute the bounds, these works show that the errors accumulated in computation form a martingale.

This paper proposes an alternative framework to characterize SR errors based on the computation of the variance. We pinpoint common error patterns in numerical algorithms and propose a lemma that bounds their variance. For each probability and through Bienaymé–Chebyshev inequality, this bound leads to better probabilistic error bound in several situations. Our method has the advantage of providing a tight probabilistic bound for all algorithms fitting our model. We show how the method can be applied to give SR error bounds for the inner product and Horner polynomial evaluation. 
\end{abstract}



\section{Introduction}
Stochastic rounding (SR) is an idea proposed in the 1950s by von Neumann and Goldstine~\cite{Neumann1947NumericalIO}. First, it can be used to estimate empirically the numerical error of computer programs; SR introduces a random noise in each floating-point operation and then a statistical analysis of the set of sampled outputs can be applied to estimate the effect of rounding errors. To make this simulation available, various tools such as Verificarlo~\cite{verificarlo}, Verrou~\cite{verrou} and Cadna~\cite{cadna} have been developed. Second, SR can be used as a replacement for the default deterministic rounding mode in numerical computations. 
It has been demonstrated that in multiple domains such as neural networks, ODEs, PDEs, and Quantum mechanics~\cite{survey}, SR provides better results compared to the IEEE-754 default rounding mode~\cite{norm}. Connolly et al.~\cite{theo21stocha} show that SR successfully prevents the phenomenon of stagnation that takes place in various applications such as neural networks, ODEs and PDEs. In particular, Gupta et al show in~\cite{gupta} that deep neural networks are prone to stagnation during the training phase. For PDEs, solved via Runge-Kutta finite difference methods in low precision, SR avoids stagnation in the computations of the heat equation solution as proved in~\cite{pde}.

Hardware units proposing stochastic rounding are still unavailable in most computers.
However, it has been introduced in various specialized processors such as Graphcore IPUs~\cite{graph-c}, which supports SR for 32 bits floating point, {\it binary32}, and 16 bits floating point, {\it binary16}, or Intel neuromorphic chip Loihi~\cite{davies2018loihi} to improve the accuracy of biological neuron and synapse models. Also, AMD~\cite{amd}, NVIDIA~\cite{nvidia}, IBM~\cite{ibm1,ibm2}, and other computing companies~\cite{com1,com2,com3} own several related patents. These developments support the idea of hardware implementations using SR becoming more available in the future.

Most current hardware implements the IEEE-754 standard, that defines five rounding modes for floating-point arithmetic which are all deterministic~\cite{norm}: round to nearest ties to even (default), round to nearest ties away, round to zero, round to $+ \infty$ and round to $-\infty$.  
SR, on the other hand, is a non-deterministic rounding mode: for a number that cannot be represented exactly in the working precision, it randomly chooses the next larger or smaller floating-point number.
In the literature, several properties and results of SR have been proven. Connolly et al show in~\cite{theo21stocha} that under SR-nearness, the expected value coincides with the exact value for a large family of algorithms.

Based on the Azuma-Hoeffding inequality and the martingale theory, recent works on the inner product~\cite{ilse} show that SR probabilistic bound of the forward error is proportional to $\sqrt{n}u$ rather than $nu$ when $nu \ll 1$. Also, the martingale central limit theorem implies that under certain conditions, the error converges in distribution to a normal distribution that is characterized by its mean and variance~\cite{dacunha2012probability}. This behavior is often observed in practice. 
In this case, the number of significant digits can be estimated by $-\log(\frac{\sigma}{\rvert \mu \lvert})$ where $\sigma$ is the standard deviation (the square root of the variance) and $\mu$ is the expected value~\cite{sohier2021confidence}. However, the results of this paper don't use any of these assumptions.

Variance also allows to use several probabilistic properties, such as concentration inequalities that provide a bound on how a random variable deviates from some value (typically, its expected value)~\cite{boucheron2013concentration}. To our knowledge, the variance analysis of a SR computation has not attracted any attention in the literature.
The purpose of this paper is to further the probabilistic investigation of SR with the following contributions:
\begin{enumerate}\addtocounter{enumi}{-1}
  \item We review the works proposed by M. P. Connolly, N. J. Higham and T. Mary~\cite{theo21stocha} and I. C. F. Ipsen, H. Zhou~\cite{ilse} that show the forward error for the inner product is proportional respectively to $\sqrt{n\ln(n)}u$ and $\sqrt{n}u$ at any probability $\lambda \leq 1$ rather than to the deterministic bound of $nu$~\cite{ilse}. 
  \item Under stochastic rounding and without any additional assumption, we propose Lemma~\ref{model}, a general framework applicable to a wide class of algorithms that allows to compute a variance bound. We choose the inner product and Horner algorithms as applications. Our bound is deterministic and depends on the condition number, the problem size $n$ and the unit roundoff of the floating-point arithmetic in use $u$.
  \item We extend the method proposed in~\cite{ilse} to derive a new forward error bound of the Horner algorithm in $O(\sqrt{n}u)$. This illustrates how these tools can be applied (with some work) to any algorithm based on a fixed sequence of sum and products.
  \item We introduce a new approach to derive a probabilistic bound in $O(\sqrt{n}u)$ based on the variance calculation and Bienaymé–Chebyshev inequality. This approach gives a tighter forward error bound than the previous bounds~\cite{theo21stocha, ilse} for a probability at most $0.758$. This bound remains tight from a rank $n$ high with respect to $u$.
\end{enumerate}

Interestingly, the variance method introduces a tight probabilistic error bound in low precision. In this regard, studying algorithms under stochastic rounding in low precision, especially bfloat-16 is becoming increasingly attractive due to its higher speed and lower energy consumption. Recent works show that in various domains such as PDEs~\cite{pde}, ODEs~\cite{ode} and neural networks~\cite{gupta}, SR provides positive effects compared to the deterministic IEEE-754~\cite{norm} default rounding mode in this precision format. 

Section~\ref{sec:back} 
presents the background on floating-point arithmetic and more particularly SR-nearness, a stochastic rounding mode introduced in \cite[p.~34]{parker1997monte}, that has the important property of being unbiased. It also satisfies the mean independence property, an assumption weaker than independence yet powerful enough to yield important results by martingale theory. 

Section~\ref{sec:var} 
is articulated around Lemma~\ref{model} that bounds the variance of the numerical error for a wide class of algorithms. We apply this result to the inner product and Horner algorithms in Theorem~\ref{variance bound-inner} and Theorem~\ref{variance-bound-horner}, respectively. 

Section~\ref{sec:pbBound} 
shows that, under SR-nearness rounding, the numerical error of these two algorithms is probabilistically bounded in $O(\sqrt{n}u)$ instead of the deterministic bound in $O(nu)$. We first prove it with the Azuma–Hoeffding inequality and martingale theory: 
we analyze techniques used for the inner product in works by Higham and Mary~\cite{theo21stocha, theo19} and Ipsen and Zhou~\cite{ilse}, point the difference in these two works, and adapt them to compute the relative error of the Horner method for polynomial evaluation. We then use the  Bienaymé–Chebyshev inequality which, combined to the previous variance bound, leads to a probabilistic bound in $O(\sqrt{n}u)$. 

The probabilistic bounds above depend on three parameters: the precision $u$, the problem size $n$, and the probability $\lambda$ that a SR-nearness computation has an error greater than the bound. In Section~\ref{bound-analyze}, we analyze these probabilistic bounds, and we show that the one obtained by the Bienaymé-Chebyshev inequality is tighter in many cases; in particular, for any given $\lambda$ and $u$, there exists a problem size $n$ above which the Bienaymé–Chebyshev bound is tighter.

Numerical experiments in Section~\ref{sec:exp} illustrate the quality of these bounds on the two aforementioned algorithms and compare them to deterministic rounding.

\section{Notations and definitions}\label{sec:back}
\subsection{Notation} 
Throughout this paper, for a random variable $X$, $E(X)$ denotes its expected value, $V(X)$ denotes its variance, and $\sigma(X)$ denotes its standard deviation. The conditional expectation of $X$ given $Y$ is $\mathbb{E}[X / Y]$.

\subsection{Floating-point background}\label{sec:FP_def}

A normal floating-point number in such a format is a number $x$ for which there exists a triple $(s, m, e)$ such that $x= \pm m \times \beta^{e-p}$, where $\beta$ is the basis, $e$ is the exponent, $p$ is the working precision, and $m$ is an integer (the significand) such that $\beta^{p-1} \leq m < \beta^p$.
We only consider normal floating-point numbers; detailed information on the floating-point format most generally in use in current computer systems is defined in the IEEE-754 norm~\cite{norm}.

Let us denote $\mathcal{F}\subset \mathbb{R}$ the set of normal floating-point numbers and $x\in \mathbb{R}$. Upward rounding $\lceil x \rceil $ and downward rounding $\lfloor x \rfloor$ are defined by:
$$ \lceil x \rceil=\min\{y\in \mathcal{F} : y \geq x\},  \quad \lfloor x \rfloor=\max\{y\in \mathcal{F} : y \leq x\},
$$
by definition, $\lfloor x \rfloor  \leq x \leq \lceil x \rceil$, with equalities if and only if $x \in\mathcal{F}$.  
The floating-point approximation of a real number $x\ne 0$ is one of $\lfloor x\rfloor$ or $\lceil x\rceil$: 
\begin{equation}
    \fl(x) =x(1+\delta), \label{fl(x)} 
\end{equation}
where $\delta = \frac{ \fl(x) - x}{x}$ is the relative error: $\lvert \delta \rvert \leq \beta^{1-p}$.
In the following, we use the same notation as~\cite{theo21stocha, ilse} $u=\beta^{1-p}$. IEEE-754 mode RN (round to nearest, ties to even) has the stronger property that $\lvert \delta \rvert \leq\frac12\beta^{1-p}=\frac12u$. In many works focusing on IEEE-754 RN, $u$ is chosen instead to be $\frac12\beta^{1-p}$.

For $x, y\in\mathcal F$, the considered rounding modes verify  $\fl(x\op y)\in\{\lfloor x\op y\rfloor, \lceil x\op y\rceil\}$ for $\op\in\{+, -, *, /\}$. Moreover, for IEEE-754 rounding modes~\cite{norm} and stochastic rounding~\cite{theo21stocha} the error in one operation is bounded:
\begin{equation}
     \fl(x \op y) = (x \op y)(1+\delta), \; \lvert \delta \rvert \leq u, \label{fl(xopy)}
 \end{equation}
specifically for RN we have $\lvert \delta \rvert \leq \frac12 u$.
Let us assume that $x$ is a real that is not representable: $x\in \mathbb{R} \setminus \mathcal{F}$.
The machine-epsilon or the distance between the two floating-point numbers enclosing $x$ is
$\epsilon(x) = \lceil x \rceil - \lfloor x \rfloor = \beta^{e-p}$. Since $\beta^{p-1} \leq m < \beta^p$, then $\beta^{e-1} \leq \lvert x \rvert < \beta^e$ and
\begin{equation}\label{epsilon-bound}
    \lvert \epsilon(x) \rvert = \beta^{e-1} u
    \leq \lvert x \rvert u.
\end{equation}

The fraction of $\epsilon(x)$ rounded away, as shown in Figure~\ref{fig:theta}, is $\theta(x) = \frac{x - \lfloor x \rfloor}{\lceil x \rceil - \lfloor x \rfloor }.$

\begin{figure}
     \centering
\begin{tikzpicture}[xscale=4]
\draw (0,0) -- (1,0);
\draw[shift={(0,0)},color=black] (0pt,0pt) -- (0pt, 2pt) node[below] {$\lfloor x \rfloor$};
\draw[shift={(1,0)},color=black] (0pt,0pt) -- (0pt, 2pt) node[below] {$\lceil x \rceil$};
\draw[shift={(.3,0)},color=black] (0pt,0pt) -- (0pt, 2pt) node[below] {$x$};
\draw[shift={(.5,0)},color=black] (0pt,0pt) -- (0pt, 2pt);
\draw[|->] (0, 30pt) -- (.5, 30pt) node[above] {$\frac{1}{2}\epsilon(x)$};
\draw[|->] (0, 5pt) -- (.3, 5pt) node[above] {$\theta(x)\epsilon(x)$};
\end{tikzpicture}
    \caption{$\theta(x)$ is the fraction of $\epsilon(x)$ to be rounded away.}
    \label{fig:theta}
\end{figure}
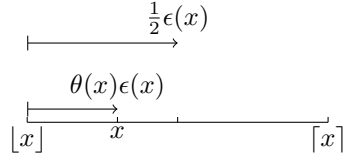

Let us denote a problem of size $n$ and precision $u$, in this paper $nu \ll 1$ means $ n \rightarrow \infty$, $ u\rightarrow 0$, and $nu \rightarrow 0$.
\subsection{Stochastic rounding definition}
Throughout this paper, $\fl(x)=\widehat x$ is the approximation of the real number $x$ under stochastic rounding. 
For $x\in \mathbb{R}\setminus \mathcal{F}$, we consider the following stochastic rounding mode, called SR-nearness:

\begin{align*}
     \fl(x) &= \left\{
     \begin{array}{cl}
        \lceil x \rceil  & \text{with probability $\theta(x)$,} \\
    \lfloor x \rfloor & \text{with probability $1-\theta(x)$.} 
          \end{array} 
          \right. 
\end{align*}
 
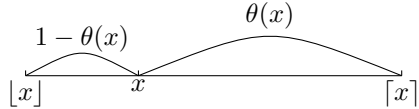
\begin{figure}
    \centering
\begin{tikzpicture}[xscale=5]
\draw (0,0) -- (1,0);
\draw[shift={(0,0)},color=black] (0pt,0pt) -- (0pt, 2pt) node[below] {$\lfloor x \rfloor$};
\draw[shift={(1,0)},color=black] (0pt,0pt) -- (0pt, 2pt) node[below] {$\lceil x \rceil$};
\draw[shift={(.3,0)},color=black] (0pt,0pt) -- (0pt, 2pt) node[below] {$x$};
\draw (0,0) .. controls (.15,.4) .. (.3,0)  (0.15,0.5) node {$1-\theta(x)$};
\draw (.3,0) .. controls (.65,.7) .. (1,0) (0.65,0.8) node {$\theta(x)$} ;
\end{tikzpicture}
    \caption{\textbf{SR-nearness}.}
\end{figure}
 
Contrary to other stochastic rounding modes~\cite{el2022positive}, SR-nearness mode is unbiased~\cite[p.~34]{parker1997monte}.
\begin{align*}
    E(\widehat x) &= \theta(x)\lceil x \rceil +(1-\theta(x))\lfloor x \rfloor \\
    &=  \theta(x)(\lceil x \rceil - \lfloor x \rfloor) + \lfloor x \rfloor =x.
\end{align*}
In the following, we focus on this stochastic rounding mode. In general and under SR-nearness, the error terms in algorithms appear as a sequence of random variables such that the independence property does not hold. However, a weaker, and yet fruitful, assumption, called mean independence, does.

\begin{defn}
A random variable $Y$ is said to be mean independent from random variable $X$ if its conditional mean $\mathbb{E}[Y / X]=\mathbb{E}(Y)$. The random sequence $(X_1, X_2, \ldots)$ is mean independent if $\mathbb{E}[X_k / X_1, . . . , X_{k-1}] = \mathbb{E}(X_k)$ for all $k$.
\end{defn}

\begin{pro}
Let $X$ and $Y$ be two real random variables:
\begin{enumerate}
    \item If $X$ and $Y$ are independent then $X$ is mean independent from $Y$.
    \item If $X$ is mean independent from $Y$ then $X$ and $Y$ are uncorrelated. 
\end{enumerate} 
The reciprocals of these two implications are false.
\end{pro}

For $a, b\in \mathcal{F}$, and $c \leftarrow a \op b$ the result of an elementary operation $\op \in \{+, -, *, /\}$ obtained from SR-nearness, the relative error $\delta$, such that
$ \widehat{c}= (a \op b)(1+\delta)
$, 
is a random variable verifying $\mathbb{E}(\delta)=0$ and $\lvert \delta\rvert\leq u$.

The following lemma has been proven in~\cite[Lem 5.2]{theo21stocha} and shows that SR-nearness satisfies the property of mean independence.

\begin{lemma}\label{meanindp}
Consider a sequence of elementary operations $c_k \leftarrow a_k \op_k b_k$, with $\delta_k$ the error of their $k$th operation, that is to say, $\widehat c_k = (\widehat a_k \op_k \widehat b_k) (1+\delta_k)$. 
The $\delta_k$ are random variables with mean zero such that $\mathbb{E}[\delta_k /  \delta_1,\ldots , \delta_{k-1}] = \mathbb{E}(\delta_k)= 0$.
\end{lemma}

\section{The variance of the error for stochastic rounding}
\label{sec:var}
We now turn to bound the variance of the error in computation.
If $\widehat x =x(1+\delta)$ is the result of an elementary operation rounded with SR-nearness, then 
$E(\widehat x)= x$ and
\begin{align*}
 V(\widehat x) &= E(\widehat x^2) - x^2 = \lceil x \rceil^2 \theta(x) + \lfloor x \rfloor^2 (1-\theta(x)) - x^2 \\
 &= \theta(x) (\lceil x \rceil^2 - \lfloor x \rfloor^2) - (x^2 - \lfloor x \rfloor^2)\\
 &= \theta(x) \epsilon(x) (\lceil x \rceil + \lfloor x \rfloor) - \theta(x) \epsilon(x)(x+ \lfloor x \rfloor)\\
 &= \theta(x) \epsilon(x) (\lceil x \rceil - x )\\
 &= \epsilon(x)^2 \theta(x) (1-\theta(x)).
\end{align*}
Using~(\ref{epsilon-bound}) leads to $V(\widehat x)\leq x^2 \frac{u^2}{4}$, in particular $V(\widehat x)\leq x^2 u^2$. Lemma~\ref{model} below allows to estimate the variance of the accumulated errors in a sequence of additions and multiplications. 
Let $K$ a subset of $\mathbb{N}$ of cardinal $n$. Assume that $\delta_1, \delta_2,...$ in that order are random errors on elementary operations obtained from SR-nearness. Let us denote 
$$ \psi_{K} = \prod_{k\in K} (1+\delta_k).
$$

Since $\lvert \delta_k \rvert \leq u$ for all $k \in K$  we have $\lvert \psi_{K} \rvert \leq (1+u)^n.$
Throughout this paper, let $\gamma_n(u)= (1+u)^{n}-1$ and $K\triangle K' = (K\cup K') \setminus (K\cap K')$. The following lemma gives some properties of $\psi$ that allow to bound the variance of errors in an algorithm consisting in a fixed sequence of sums and products.

\begin{lemma}\label{model}
Under SR-nearness $\psi_{K}$ satisfies 
\begin{enumerate}
    \item $E(\psi_{K}) = 1$.
    \item Let $K' \subset \mathbb{N}$ such that $\lvert K\cap K'\rvert = m$, under the assumption that $\forall~j\in~K\triangle~K'$, $k\in K\cap K'$, with $j<k$ we have
    $$  0 \leq \Cov(\psi_{K},\psi_{K'})  \leq \gamma_{m}(u^2). 
    $$  
    \item $V(\psi_K) \leq \gamma_n(u^2) $,
\end{enumerate}
where $\gamma_n(u^2)= (1+u^2)^{n}-1 \approx \exp{(nu^2)} -1 = nu^2 + O(u^4).$
\end{lemma}

\begin{proof}
The first point is an immediate consequence of~\cite[lem 6.1]{theo21stocha}. The third point is a particular case of the second with $K=K'$. Let us prove point $2$. 
$$\Cov(\psi_{K},\psi_{K'}) = E(\psi_{K} \psi_{K'}) -E(\psi_{K})E(\psi_{K'}) = E(\psi_{K} \psi_{K'}) -1.
$$
Assume that $K\cap K' = \{k_1,...,k_m\}$. Let us denote
$$Q_m :=\psi_{K} \psi_{K'} =  \prod_{j\in K\triangle K'} (1+\delta_j) \prod_{l=k_1}^{k_m} (1+\delta_l)^2, $$
such that $j < k_i$ for all $j\in K\triangle K'$ and $i \in\{1,...,m\}$.

We prove by induction over $m$ that $1 \leq E(Q_m ) \leq (1+u^2)^m$. For $m=0$, we have $K\cap K'=\emptyset$ and $Q_0 = \prod_{j\in K\triangle K'} (1+\delta_j)$, from the first point $E(Q_0)=1$.
Assume that the inequality holds for $Q_{m-1}$.
\begin{align*}
Q_m &= (1+\delta_{k_m})^2\prod_{l =k_1}^{k_{m-1}} (1+\delta_l)^2 \prod_{j\in K\triangle K'} (1+\delta_j) = (1+\delta_{k_m})^2 Q_{m-1}.
\end{align*}
Let us denote $ \mathcal{S}_{K\triangle K'}=\{ \delta_j, \ j\in K\triangle K')\}$, using the law of total expectation $E(X)= E(E[X/Y])$ and lemma~\ref{meanindp} we have
\begin{align*}
E(Q_m) &= E\big(  (1+\delta_{k_m})^2 Q_{m-1} \big) = E\big(  E[ (1+\delta_{k_m})^2 Q_{m-1}/ \mathcal{S}_{K\triangle K'}, \delta_{k_1},...,\delta_{k_{m-1}}] \big)\\
    &= E\big( Q_{m-1}  E[(1+\delta_{k_m})^2/ \mathcal{S}_{K\triangle K'}, \delta_{k_1},...,\delta_{k_{m-1}}]\big)\\
     &=  E\big(  Q_{m-1} E[1+\delta_{k_m}^2/\mathcal{S}_{K\triangle K'}, \delta_{k_1},...,\delta_{k_{m-1}}]\big).
\end{align*}
Since $\lvert \delta_{k_m}\rvert \leq u$, we have
\begin{align*}
E(Q_{m-1}) \leq E\big(Q_{m-1} E[1+\delta_{k_m}^2/\mathcal{S}_{K\triangle K'}, \delta_{k_1},...,\delta_{k_{m-1}}]\big) \leq E\big(Q_{m-1}(1+u^2)\big).
\end{align*}
Thus, $1 \leq E\big(  Q_{m}\big) \leq (1+u^2)^m.$ Finally, by induction, the claim is proven
\begin{align*}
    0\leq E\big(  Q_{m}\big) -1 =   \Cov(\psi_{K},\psi_{K'})  \leq \gamma_m(u^2). 
\end{align*}

\end{proof}

Under SR-nearness, Lemma~\ref{model} can now be used to derive a variance bound for many algorithms, such as inner products, matrix-vector and matrix-matrix products, solutions of triangular systems, and the Horner algorithm. In the following, we chose the inner product and Horner algorithms as applications.

\subsection{Inner product}

Consider the inner product $s_n = y = a_1 b_1 + \ldots + a_n b_n$, evaluated from left to right, i.e, $s_i = s_{i-1} + a_i b_i$, starting with $s_1 = a_1b_1$. 
In this paper, we address the sequential method which has a deterministic bound proportional to $nu$. However, the accumulator implementation of the inner product using a binary tree leads to a deterministic error bound in $O(\ln{(n)}u)$.

Let $\delta_0=0$, the computed $\widehat s_i$ satisfies $\widehat s_1 =a_1b_1(1+\delta_1)$ and
$$ \widehat s_i =(\widehat s_{i-1} +a_ib_i(1+\delta_{2i-2}))(1+\delta_{2i-1}), \quad \lvert \delta_{2i-2} \rvert, \lvert \delta_{2i-1} \rvert \leq u,
$$
for all $2\leq i \leq n$, where $\delta_{2i-2}$ and $\delta_{2i-1}$ represent the rounding errors from the products and additions, respectively. We thus have
\begin{equation*}
     \widehat y = \widehat s_n = \sum_{i=1}^{n} a_ib_i(1+\delta_{2i-2}) \prod_{k=i}^n (1+\delta_{2k-1}).
\end{equation*}

\begin{theorem}\label{variance bound-inner}
Under SR-nearness, the computed $\widehat y$ satisfies $ E(\widehat y)=y$ and
\begin{equation}\label{innervar}
    V(\widehat y) \leq y^2  \mathcal{K}_1^2 \gamma_n(u^2),
\end{equation}

where $\mathcal{K}_1 =\frac{\sum_{i=1}^{n} \lvert a_ib_i \rvert}{\lvert \sum_{i=1}^{n}  a_ib_i \rvert}$ is the condition number using the 1-norm for the computed  $y=\sum_{i=1}^n a_ib_i$. 
\end{theorem}

\begin{proof}

For all $1 \leq i \leq n$, we have
$$ \widehat y = \sum_{i=1}^{n} a_ib_i(1+\delta_{2i-2}) \prod_{k=i}^n (1+\delta_{2k-1}) = \sum_{i=1}^{n} a_ib_i \psi_{K_i},
$$ 
with $K_i = \{2i-2, 2i-1, 2i+1,\ldots,2n-1\}$. 
Lemma~\ref{model} shows that $E(\psi_{K_i}) = 1$ for all $1\leq i \leq n$, hence
$$ E(\widehat y)= E\big(\sum_{i=1}^{n} a_ib_i \psi_{K_i} \big)= \sum_{i=1}^{n} a_ib_i E(\psi_{K_i})= y.
$$ 
For all $1\leq i < j \leq n$, 
$K_j  \cap K_i =  \{2j-1, 2j+1,...,2n-1\}$ and $\card(K_j  \cap K_i) = n-j +1$.
\begin{align*}
V(\widehat y) &= V\big( \sum_{i=1}^{n} a_ib_i \psi_{K_i} \big)\\
& \leq \left( \sum_{i=1}^{n}  \lvert a_ib_i \rvert \sqrt{V(\psi_{K_i})} \right)^2   && \text{since $\sigma(X+Y) \leq \sigma(X) + \sigma(Y)$} \\
&\leq  \left( \sum_{i=1}^{n}  \lvert a_ib_i \rvert \sqrt{ \gamma_{n-i+1}(u^2)} \right)^2 && \text{by Lemma~\ref{model}} \\
&\leq \gamma_n(u^2) (\sum_{i=1}^{n} \lvert a_ib_i \rvert)^2 && \text{since $\gamma_{n-i+1}(u^2) \leq \gamma_{n}(u^2)$} \\
&=  y^2 \mathcal{K}_1^2 \gamma_n(u^2).
\end{align*}

\end{proof}

\begin{remark}
Because $E(\widehat y) = y$, under a normality assumption, the number of significant bits can be lower-bounded by
\begin{align*}
   -\log_{2}\left( \frac{\sigma(\widehat y)}{\lvert E(\widehat y) \rvert}\right) & \geq  -\log_{2}\left( \mathcal{K}_1 \sqrt{\gamma_n(u^2)}\right) \approx -\log_{2} (\mathcal{K}_1) - \log_{2}(u) -\frac12 \log_{2}(n).
\end{align*}
\end{remark}

\subsection{Horner algorithm}
Horner algorithm is an efficient way of evaluating polynomials. When performed in floating-point arithmetic, this algorithm may suffer from catastrophic cancellations and yield a computed value less accurate than expected.

\begin{mode}\label{mod}
Let $P(x) = \sum_{i=0}^n a_i x^i$, Horner rule consists in writing this polynomial as 
$$P(x)= (((a_nx +a_{n-1})x +a_{n-2})x \ldots +a_1)x +a_0.
$$

We define by induction the following sequence \\

\begin{center}
   {\renewcommand{\arraystretch}{1.3}

    \begin{tabular}{|C{2.5cm}||L{5.5cm}|L{3.5cm}|}
\hline Operation & Floating-point arithmetic & Exact computation  \\
\hline    & $\widehat r_0 = a_n $ &  $r_0 = a_n $  \\
         $* $    &  $\widehat{r}_{2k-1}=\widehat{r}_{2k-2} x (1+\delta_{2k-1}) $ &  $r_{2k-1} = r_{2k-2}x $\\
          $+ $   &  $\widehat{r}_{2k} = (\widehat r_{2k-1} +a_{n-k})(1+\delta_{2k}) $ &  $r_{2k} = r_{2k-1} +a_{n-k} $\\
\hline Output &  $\widehat{r}_{2n}=\widehat{P}(x) $ &  $r_{2n}= P(x) $\\
\hline

\end{tabular}
}
\end{center}
          
\end{mode}
for all $1 \leq k \leq n$, with $\delta_{2k-1}$ and $\delta_{2k}$ the rounding errors from the products and the additions respectively. Let $\delta_0 =0$, we thus have
$$ \widehat r_{2n} = \sum_{i=0}^{n}a_i x^i \prod_{k=2(n -i)}^{2n} (1+\delta_k).
$$

\begin{theorem}\label{variance-bound-horner}
Using SR-nearness, the computed $\widehat r_{2n}$ satisfies $E(\widehat r_{2n}) = r_{2n}$ and
\begin{equation}\label{hor-var}
         V(\widehat r_{2n}) \leq r_{2n}^2 \mathcal{K}_1^2 \gamma_{2n}(u^2),
\end{equation}
where $\mathcal{K}_1 =\frac{\sum_{i=0}^{n} \lvert a_i x^i \rvert}{\lvert \sum_{i=0}^{n}  a_i x^i \rvert}$ is the condition number using the 1-norm for the computed  $P(x)=\sum_{i=0}^n a_i x^i$. 
\end{theorem}

\begin{proof}
We have
$$\widehat r_{2n} = \sum_{i=0}^{n}a_i x^i \prod_{k=2(n -i)}^{2n} (1+\delta_k)
 =  \sum_{i=0}^{n}a_i x^i \psi_{K_i},$$ 
with $K_i =\{ 2(n-i), 2(n-i)+1,...,2n\}$ for all $0 \leq i \leq n$.
Lemma~\ref{model} implies $E(\psi_{K_i}) = 1$, then
$E(\widehat r_{2n}) = E\big( \sum_{i=0}^{n}a_i x^i \psi_{K_i}\big)
    = \sum_{i=0}^{n}a_i x^i E(\psi_{K_i}) = r_{2n}.$
Therefore, because $\delta_0 =0$ we have
   \begin{align*}
       V(\widehat r_{2n}) &= V\big(\sum_{i=0}^{n}a_i x^i \psi_{K_i} \big) \\
       & \leq  \left( \sum_{i=0}^{n}  \lvert a_i x^i \rvert \sqrt{V(\psi_{K_i})} \right)^2   && \text{since $\sigma(X+Y) \leq \sigma(X) + \sigma(Y)$} \\
       &\leq  \left( \sum_{i=0}^{n}  \lvert a_i x^i \rvert \sqrt{ \gamma_{2i}(u^2)} \right)^2 && \text{by Lemma~\ref{model}} \\
       &\leq \gamma_{2n}(u^2) (\sum_{i=0}^{n} \lvert a_i x^i \rvert)^2 && \text{since $\gamma_{2i}(u^2) \leq \gamma_{2n}(u^2)$} \\
       &= r_{2n}^2 \mathcal{K}_1^2 \gamma_{2n}(u^2).
   \end{align*}

\end{proof}

\begin{remark}
Because $E(\widehat r_{2n}) = r_{2n}$, under a normality assumption, the number of significant bits can be lower-bounded by
\begin{align*}
    -\log_{2}\left( \frac{\sigma(\widehat r_{2n})}{\lvert E(\widehat r_{2n}) \rvert}\right) & \geq  -\log_{2}\left( \mathcal{K}_1 \sqrt{\gamma_{2n}(u^2)}\right)\\
    &\approx -\log_{2} (\mathcal{K}_1) - \log_{2}(u) -\frac12 \log_{2}(2n).
\end{align*}
\end{remark}

\section{Probabilistic bounds of the error for stochastic rounding}
\label{sec:pbBound}

Based on the independence assumption, Higham and Mary~\cite{theo19} have shown that for the inner product,  a probabilistic bound of the error proportional to $\sqrt{n\ln{(n)}}u$ can be achieved rather than the deterministic bound in $O(nu)$. With Connelly, they show in~\cite{theo21stocha} that this bound always holds for SR-nearness due to mean independence of errors.

We start with the approaches based on the Azuma-Hoeffding inequality and the martingale property (AH1 and AH2 methods in the following). In this context, firstly, we give a rigorous review of the previous results of the inner product forward error by Higham and Mary~\cite{theo21stocha} and Ilse, Ipsen, and Zhou~\cite{ilse}. Secondly, we extend these techniques to the Horner algorithm, which also gives a probabilistic bound proportional to $\sqrt{n}u$. 

Then, we present a new approach based on Bienaymé–Chebyshev inequality and the previous variance estimations (BC method in the following), our bound is also in $O(\sqrt{n} u)$ and it is lower than the AH1 and AH2 bounds in several situations for both inner product and Horner algorithms.

\subsection{Azuma-Hoeffding method}\label{azuma-method}
Let us recall the concept of a martingale and the Azuma-Hoeffding inequality for a martingale~\cite{azum}. 

\begin{defn}
\label{def:martingale}
A sequence of random variables $M_1, ... , M_n$ is a martingale with respect to the sequence $X_1, . . . , X_n$ if, for all $k,$
\begin{itemize}
    \item $M_k$ is a function of $X_1, ..., X_k$,
    \item $\mathbb{E}(\lvert M_k \rvert ) < \infty,$ and
    \item $\mathbb{E}[M_k /  X_1, ..., X_{k-1}]=M_{k-1}$.
\end{itemize}

\end{defn}

\begin{lemma}(Azuma-Hoeffding inequality). Let $M_0, ..., M_n$ be a martingale with respect to a sequence $X_1, . . . , X_n.$ We assume that there exist $a_k<b_k$ such that $a_k \leq M_k - M_{k-1} \leq  b_k$ for $k = 1: n.$ Then, for any $A  > 0$ 
$$ \mathbb{P}(\lvert M_n - M_0 \rvert \geq A) \leq 2 \exp \left( 
-\frac{2A^2}{\sum_{k=1}^n(b_k-a_k)^2} 
\right).
$$
In the particular case $a_k=-b_k$ and $\lambda = 2 \exp \left( 
-\frac{A^2}{2\sum_{k=1}^n b_k^2} \right) $ we have 
$$ \mathbb{P}\left( \lvert M_n - M_0 \rvert \leq \sqrt{\sum_{k=1}^n b_k^2} \sqrt{2 \ln (2 / \lambda)} \right) \geq 1- \lambda,
$$
where $0< \lambda <1$.
\end{lemma}

\subsubsection{Inner product}\label{inner-product}

Under SR-nearness, the inner product $y=a^{\top}b,$ where $a,b\in \mathbb{R}^n$ is defined as
$\widehat y = \widehat s_n =\sum_{i=1}^{n} a_ib_i(1+\delta_{2(i-1)}) \prod_{k=i}^n (1+\delta_{2k-1}).$ The worst case of the forward error of the computed $\widehat y$ is in $O(n u)$. Wilkinson~\cite[sec 1.33]{wilk} had the intuition that the roundoff error accumulated in $n$ operations is proportional to $\sqrt{n} u$ rather than $n u$. Based on the mean independence of errors established in Lemma~\ref{meanindp}, Connelly et al.~\cite{theo21stocha} and Ilse, Ipsen and Zhou~\cite{ilse} have investigated this problem for SR-nearness. Both works build on the mean independence property of SR-nearness. This allows them to form a martingale, and then to apply the Azuma-Hoeffding concentration inequality. The difference between these two works is in the way they form the martingale. In \cite[sec 3]{theo21stocha}, the martingale is built using the errors accumulated in the whole process $\psi_{K_i}=(1+\delta_{2(i-1)}) \prod_{k=i}^n (1+\delta_{2k-1})$ for all $ 1\leq i \leq n$. Azuma-Hoeffding inequality implies that $\lvert \psi_{K_i} \rvert \leq \tilde{\gamma}_n(\lambda)$ with probability at least $1-2\exp{\frac{-\lambda^2}{2}}$, where $\tilde{\gamma}_n(\lambda) = \exp{\frac{\lambda\sqrt{n}u + nu^2}{1-u}} -1.$ This approach uses the inclusion-exclusion principle to generalize the bound to the summation, which results in a pessimistic $n$ in the probability. They prove 
$$ \frac{\lvert \widehat y  - y \rvert}{\lvert y \rvert}    \leq \mathcal{K}_1 \tilde{\gamma}_n(\lambda),
$$
with probability at least $1-2n\exp{\frac{-\lambda^2}{2}}$. The factor $n$ in the probability disrupts the $\sqrt{n}u$ property. $\delta = 2n\exp{\frac{-\lambda^2}{2}}$ implies that $\lambda = \sqrt{2\ln{(2n/\delta)}}$ and
\begin{align}
    \frac{\lvert \widehat y  - y \rvert}{\lvert y \rvert} \leq \mathcal{K}_1 \tilde{\gamma}_n\big(\sqrt{2\ln{(2n/\delta)}} \big), \label{higham-bound-inn} \tag{AH1-IP} 
\end{align}
with probability at least $1-\delta$. When $nu \ll 1$, we have
\begin{align*}
    \tilde{\gamma}_n(\sqrt{2\ln{(2n/\delta)}}) &= \exp{\frac{\sqrt{2n\ln{(2n/\delta)}}u + nu^2}{1-u}} -1 \\
    &= u\sqrt{2n\ln{(2n/\delta)}}  + O(u^2)\\
    &= u\sqrt{2n\ln{2n} -2n\ln{\delta}}  + O(u^2)= O(u\sqrt{n\ln{n})}.
\end{align*}

On the other hand,~\cite[sec 4]{ilse} forms it by following step-by-step how the error accumulates in the recursive summation of the inner product. In particular, they distinguish between the multiplications and additions computed at each step and carefully monitor their mean independences. This approach leads to the following probabilistic bound
\begin{equation}
    \frac{\lvert \widehat y  - y \rvert}{\lvert y \rvert}    \leq \mathcal{K}_1 \sqrt{u \gamma_{2n}(u)} \sqrt{\ln (2 / \delta)}, \label{ilse-bound-inner}\tag{AH2-IP}
\end{equation}
with probability at least $1-\delta$. This technique avoids the inclusion-exclusion principle and when $nu \ll 1$, it leads to 
$$ \sqrt{u \gamma_{2n}(u)} \sqrt{\ln (2 / \delta)} = u\sqrt{2n\ln{2} -2n \ln{\delta}} +O(u^2).
$$

Note that when $nu \ll 1$,~(\ref{higham-bound-inn}) and~(\ref{ilse-bound-inner}) differ only in the factor $\sqrt{\ln{n}}$ that appears in~(\ref{higham-bound-inn}) due to the use of the martingale property on each partial sum. All in all, ~(\ref{ilse-bound-inner}) is proportional to $u\sqrt{n}$, while~(\ref{higham-bound-inn}) is proportional to $u\sqrt{n\ln{n}}$. An analysis of the case $nu \gg 1$ will be presented in Section~\ref{bound-analyze}. 

\subsubsection{Horner algorithm}

In the following, we derive a probabilistic bound for the computed $\widehat{P}(x)$ based on the previous method applied for the inner product in~\cite[sec 4]{ilse}.

With the notations defined in Model~\ref{mod}, let us denote $Z_i :=\widehat r_i - r_i$ for all $0 \leq i \leq 2n$. The total forward error is $\lvert Z_{2n} \rvert = \lvert \widehat r_{2n} -  r_{2n} \rvert = \lvert \widehat{P}(x)  - P(x) \rvert$ and 

\begin{align*}
\lvert \widehat{P}(x)  - P(x) \rvert &=  \left\lvert \sum_{i=0}^n a_i x^{i} \left( \prod_{k=2(n -i)}^{2n} (1+\delta_k)-1 \right)  \right\rvert 
\leq \sum_{i=0}^n \lvert  a_i x^i\rvert  \gamma_{2n}(u).
\end{align*}
Finally, 

\begin{equation}\label{detbound}
    \frac{\lvert \widehat{P}(x)  - P(x) \rvert}{\lvert P(x) \rvert} \leq \mathcal{K}_1 \gamma_{2n}(u).
\end{equation}
The deterministic bound is proportional to $nu$. In the following, we prove a probabilistic bound in $O(\sqrt{n}u)$.

The partial sum forward errors satisfy
\begin{align*}
Z_{2k-1} &= \widehat r_{2k-1} - r_{2k-1} = \widehat r_{2k-2}x (1+\delta_{2k-1}) - r_{2k-2}x\\
         &= xZ_{2k-2} +\widehat r_{2k-2}x \delta_{2k-1},  \\
 Z_{2k}  &= \widehat r_{2k} - r_{2k} = (\widehat r_{2k-1} +a_{n-k})(1+\delta_{2k}) - r_{2k-1} - a_{n-k}\\
         &= Z_{2k-1} +(\widehat r_{2k-1} +a_{n-k}) \delta_{2k}, 
\end{align*}
for all $1 \leq k \leq n$.
The sequence $Z_0,...,Z_{2n}$ does not form a martingale with respect to $\delta_1,...,\delta_{2n}$ due to the multiplication in odd steps, $$ E[Z_{2k-1}/\delta_1,...,\delta_{2k-2}]= xZ_{2k-2}.$$
In order to form a martingale and use the Azuma-Hoeffding inequality, we define the following variable change  
$$ Y_i = \frac{Z_i}{x^{\llfloor (i+1)/2 \rrfloor}},$$
where $\llfloor (i+1)/2 \rrfloor$ is the integer part of $(i+1)/2$, we thus have
  \begin{align}
      \left\{
     \begin{array}{ccl}
         Y_{2k-1} &=& Y_{2k-2} + \frac{1}{x^{k-1}}\widehat r_{2k-2} \delta_{2k-1}, \\
         Y_{2k}  &=& Y_{2k-1} + \frac{1}{x^k}(\widehat r_{2k-1} + a_{n-k})\delta_{2k},
          \end{array} 
          \right. 
\label{eq2} \end{align}
for all $1 \leq k \leq n$ with $Y_0 = 0$.

\begin{theorem}
The sequence of random variables $Y_0,...,Y_{2n}$ is a martingale with respect to $\delta_1, ..., \delta_{2n}$.
\end{theorem}

\begin{proof}
We check that the three conditions of Definition~\ref{def:martingale} are satisfied. Throughout ~the ~proof, ~we ~note ~the ~set ~$\mathbb{F}_k= \{\delta_1, ..., \delta_k\}$.

\begin{itemize}
   \item The recursion in Model~\ref{mod} shows that $Y_i$ is a function of $\delta_1, ..., \delta_{i}$ for all $1 \leq i \leq 2n$.
   
   \item $\mathbb{E}(\lvert Y_i \rvert ) $ is finite because $x$ and $a_k$ are finite ~for ~all ~$n-i \leq k \leq n$ and $\lvert \delta_j \rvert \leq u$ for all $1 \leq j \leq i$.
    
    \item We prove that $\mathbb{E}[Y_i /\mathbb{F}_{i-1}] = Y_{i-1}$ by distinguishing the even and odd cases. 
    Firstly, using the mean independence of $\delta_1, ... \delta_{2k-1}$ and Equation~(\ref{eq2}) we obtain
    \begin{align*}
    \mathbb{E}[Y_{2k-1}/\mathbb{F}_{2k-2}] &= \mathbb{E}[Y_{2k-2}/\mathbb{F}_{2k-2}] 
     + \mathbb{E}[ \frac{1}{x^{k-1}}\widehat r_{2k-2} \delta_{2k-1}/\mathbb{F}_{2k-2}]\\
    &= Y_{2k-2} 
     + \frac{1}{x^{k-1}}\widehat r_{2k-2} \mathbb{E}[\delta_{2k-1}/\mathbb{F}_{2k-2}] = Y_{2k-2}.
    \end{align*}
    \end{itemize}
    
Secondly, using the mean independence of $\delta_1, ... \delta_{2k}$ and Equation~(\ref{eq2}) we obtain
    \begin{align*}
    \mathbb{E}[Y_{2k}/\mathbb{F}_{2k-1}] &= \mathbb{E}[Y_{2k-1}/\mathbb{F}_{2k-1}] + \mathbb{E}[\frac{1}{x^k}(\widehat r_{2k-1} + a_{n-k})\delta_{2k}/\mathbb{F}_{2k-1}]\\
    &= Y_{2k-1} + \frac{1}{x^k}(\widehat r_{2k-1} + a_{n-k})\mathbb{E}[\delta_{2k}/\mathbb{F}_{2k-1}]= Y_{2k-1}.
    \end{align*}
\end{proof}

\begin{lemma}\label{cst-bound}
The above martingale $Y_0,..., Y_{2n}$ satisfies
$ \lvert Y_i - Y_{i-1} \rvert \leq C_i u$, for all $1\leq i \leq 2n,$
where
 \begin{align*}
      \left\{
     \begin{array}{ccl}
       C_{2k-1} &=&  \lvert a_n \rvert (1+u)^{2k-2} + \sum_{j=1}^{k-1} \lvert a_{n-j} \rvert \lvert x \rvert^{-j}(1+u)^{2(k-j)-1},\\
       C_{2k} &=& \lvert a_n \rvert (1+u)^{2k-1} + \sum_{j=1}^k \lvert a_{n-j} \rvert \lvert x \rvert^{-j}(1+u)^{2(k-j)},
\end{array} 
          \right. 
 \end{align*}
for all $1\leq k \leq n.$
\end{lemma}

\begin{proof}
Note that $Y_0=0$, then $\lvert Y_1 - Y_0 \rvert = \lvert Y_1 \rvert = \lvert a_n \rvert $ and the equality holds for $C_1$. Using Equation~(\ref{eq2})
$$ \lvert Y_{2k-1} - Y_{2k-2} \rvert  \leq \frac{1}{\lvert x \rvert^{k-1}} \lvert \widehat r_{2k-2} \rvert u.
$$
Moreover,
\begin{align*}
\lvert \widehat r_{2k-2} \rvert  &\leq \lvert \widehat r_{2k-3} \rvert (1+u) + \lvert a_{n-k+1}\rvert (1+u) \leq \lvert \widehat r_{2k-4} \rvert \lvert x \rvert (1+u)^2 + \lvert a_{n-k+1}\rvert (1+u),
\end{align*}
by induction we obtain
$$ \lvert \widehat r_{2k-2} \rvert \leq \lvert a_n \rvert  \lvert x \rvert^k (1+u)^{2k-2} + \sum_{j=1}^{k-1} \lvert a_{n-j} \rvert \lvert x \rvert^{k-j}(1+u)^{2(k-j)-1}.
$$
This completes the proof for $C_{2k-1}$ for all $1\leq k \leq n$. A similar approach can be applied to prove the same result for $C_{2k}$ for all $1\leq k \leq n$.
\end{proof}
We now have all the tools to state and demonstrate the main result of this section:

\begin{theorem}
Under SR-nearness, for all $0 < \lambda <1$ and with probability at least $1-\lambda$
\begin{equation}
    \frac{\lvert \widehat{P}(x)   - P(x) \rvert}{\lvert P(x) \rvert} \leq \mathcal{K}_1  \sqrt{u \gamma_{4n}(u)} \sqrt{\ln (2 / \lambda)},
\end{equation}
where $\mathcal{K}_1 = \frac{\sum_{i=0}^n \lvert a_i x^i \rvert}{\lvert P(x) \rvert}$ is the condition number of the polynomial evaluation and $\gamma_{4n}(u)=(1+u)^{4n} -1$.
\end{theorem}

\begin{proof}
Recall that $\lvert \widehat r_{2n} - r_{2n} \rvert = \lvert Z_{2n} \rvert = \lvert x^n \rvert \lvert Y_{2n} \rvert$. Therefore, $ Y_0,...,Y_{2n}$ is a martingale with respect to $\delta_1, ..., \delta_{2n}$ and Lemma~\ref{cst-bound} implies $ \lvert Y_i - Y_{i-1} \rvert \leq C_i u$ for all $1\leq i \leq 2n$. 
Using the Azuma-Hoeffding inequality yields
$$\mathbb{P}\left( \lvert Y_{2n} \rvert \leq u \sqrt{\sum_{i=1}^{2n}C_i^2}\sqrt{2 \ln (2 / \lambda)}\right) \geq 1-\lambda,
$$
it follows that
$$ \lvert Z_{2n} \rvert \leq u \sqrt{\sum_{i=1}^{2n} (\lvert x \rvert^n C_i)^2}\sqrt{2 \ln (2 / \lambda)}, 
$$
with probability at least $1-\lambda$, where
\begin{align*}
\lvert x \rvert^n C_{2k} &= \lvert a_n \rvert \lvert x \rvert^n (1+u)^{2k-1} + \sum_{j=1}^k \lvert a_{n-j} x^{n-j} \rvert(1+u)^{2(k-j)}\\
&\leq(1+u)^{2k-1} \sum_{j=0}^k \lvert a_{n-j} x^{n-j} \rvert \leq(1+u)^{2k-1} \sum_{j=0}^n \lvert a_{j} x^j \rvert,
\end{align*}
for all $1\leq k \leq n.$
Hence, $$ (\lvert x \rvert^n C_{2k})^2 \leq (1+u)^{2(2k-1)} \big(\sum_{j=0}^n \lvert a_{j} x^j \rvert \big)^2.
$$
In a similar way,
$$ (\lvert x \rvert^n C_{2k-1})^2 \leq (1+u)^{2(2k-2)} \big(\sum_{j=0}^n \lvert a_{j} x^j \rvert \big)^2.
$$
Thus, 
\begin{align*}
\sum_{i=1}^{2n} (\lvert x \rvert^n C_i)^2 &\leq \big(\sum_{j=0}^n \lvert a_{j} x^j \rvert \big)^2 \sum_{i=0}^{2n-1} ((1+u)^{2})^i\\
&= \big(\sum_{j=0}^n \lvert a_{j} x^j \rvert \big)^2 \frac{((1+u)^2)^{2n}-1}{(1+u)^2-1}=  \big(\sum_{j=0}^n \lvert a_{j} x^j \rvert \big)^2  \frac{\gamma_{4n}(u)}{u^2+2u}.
\end{align*}
As a result, 
$$ \lvert \widehat{P}(x)  - P(x) \rvert = \lvert Z_{2n} \rvert \leq \sum_{j=0}^n \lvert a_{j} x^j \rvert  \sqrt{\frac{u \gamma_{4n}(u)}{2+u}}  \sqrt{2 \ln (2 / \lambda)},
$$
with probability at least $1-\lambda$.
Finally,
$$
    \frac{\lvert \widehat{P}(x)  - P(x) \rvert}{\lvert P(x) \rvert} \leq \mathcal{K}_1 \sqrt{u \gamma_{4n}(u)} \sqrt{\ln (2 / \lambda)},
$$
with probability at least $1-\lambda$.
\end{proof}

\subsection{Bienaymé–Chebyshev method}

Another way to obtain a probabilistic $O(\sqrt{n}u)$ bound is to use Bienaymé–Chebyshev inequality. This method requires only information on the variance. Moreover, we will see in Section~\ref{bound-analyze} that for any probability $\lambda$ there exists $n$ such that this method introduces a tighter probabilistic bound than the Azuma-Hoeffding method.

\begin{lemma}(Bienaymé–Chebyshev inequality)\label{Bienaymé–Chebyshev-inequality}
Let $X$ be a random variable with finite expected value and finite non-zero variance. For any real number $\alpha > 0$,
$$ \mathbb{P}\big(\lvert X - E(X) \rvert \leq \alpha \sqrt{V(X)}\big) \geq 1- \frac{1}{\alpha^2}.
$$
\end{lemma}
Regarding the two algorithms above, the computed $\widehat{y}$ satisfies $E(\widehat y)= y$, then
$$\mathbb{P}\left(
\lvert \widehat{y} - y \rvert \leq \alpha \sqrt{V(\widehat y)} 
\right) \geq 1-\frac{1}{\alpha^2},
$$
taking $\lambda = \frac{1}{\alpha^2}$ yields 
$\lvert \widehat{y} - y  \rvert \leq  \sqrt{V(\widehat y)/\lambda}$ with probability at least $1-\lambda$.
\subsubsection{Inner product}\label{sub-inner}

From Theorem \ref{innervar} we have 
$$ \frac{\sqrt{V\big( \widehat y) /\lambda}}{\lvert y \rvert} \leq \mathcal{K}_1 \sqrt{\gamma_{n}(u^2)/\lambda}.
$$
Thus,
$$ \frac{\lvert \widehat{y} - y  \rvert}{\lvert y  \rvert} \leq \frac{\sqrt{V\big( \widehat y)/\lambda}}{\lvert y \rvert} \leq \mathcal{K}_1 \sqrt{\gamma_{n}(u^2)/\lambda},
$$
and
\begin{equation}
    \mathbb{P} \left( \frac{\lvert \widehat{y} - y  \rvert}{\lvert y  \rvert}  \leq  \mathcal{K}_1 \sqrt{\gamma_{n}(u^2)/\lambda}\right) \geq  \mathbb{P} \left( \frac{\lvert \widehat{y} - y  \rvert}{\lvert y  \rvert}  \leq  \frac{\sqrt{V\big( \widehat y) /\lambda}}{\lvert y \rvert} \right) \geq 1-\lambda.
\end{equation}

\subsubsection{Horner algorithm}

From Theorem \ref{hor-var} we have 
$$ \frac{V\big( \widehat{P}(x) \big)}{\lvert P(x) \rvert} \leq \mathcal{K}_1 \sqrt{\gamma_{2n}(u^2)}.
$$
The previous reasoning from Sub-section~\ref{sub-inner} leads to
\begin{equation}
    \mathbb{P} \left( \frac{\lvert \widehat{P}(x) - P(x) \rvert}{\lvert P(x) \rvert}  \leq  \mathcal{K}_1 \sqrt{\gamma_{2n}(u^2)/\lambda} \right) \geq1-\lambda.
\end{equation}

\section{Bounds analysis}\label{bound-analyze}

In the following, we compare the various bounds of the two previous algorithms and analyze which bound is the tightest depending on the precision in use, the target probability, and the number of operations.
\subsection{Inner product}\label{bound-compare-inner}
In the beginning, let us recall all bounds obtained for the inner product $y=a^{\top}b,$ where $a,b\in\mathbb{R}^n$
\begin{align}
\frac{\lvert \widehat{y}  - y \rvert}{\lvert y \rvert} &\leq \mathcal{K}_1 \gamma_{n}(u), \quad  \label{det-inn-bound}\tag{Det-IP}\\
\frac{\lvert \widehat{y}  - y \rvert}{\lvert y \rvert}    &\leq \mathcal{K}_1 \tilde{\gamma}_n(\sqrt{2\ln{(2n/\lambda)}}) && \text{with probability at least $1-\lambda$,} \label{higham-inn-bound}\tag{AH1-IP} \\
\frac{\lvert \widehat{y}  - y \rvert}{\lvert y \rvert}    &\leq \mathcal{K}_1 \sqrt{u \gamma_{2n}(u)} \sqrt{\ln (2 / \lambda)} && \text{with probability at least $1-\lambda$}, \label{ilse-inn-bound}\tag{AH2-IP} \\
  \frac{\lvert \widehat{y}  - y \rvert}{\lvert y \rvert}   &\leq \mathcal{K}_1 \sqrt{ \gamma_{n}(u^2) } \sqrt{1/ \lambda} &&\text{with probability at least $1-\lambda$},  \label{cheb-inn-bound}\tag{BC-IP}
\end{align}
where $ \tilde{\gamma}_n(\sqrt{2\ln{(2n/\lambda)}}) = \exp{\left( \frac{\sqrt{2n \ln{(2n/\lambda)}} u + nu^2}{1-u}
\right)}-1$. 

All bounds have the same condition number $\mathcal{K}_1$, but differ in the others factor: $\gamma_{n}(u)$ for~(\ref{det-inn-bound}), $\tilde{\gamma}_n(\sqrt{2\ln{(2n/\lambda)}})$ for~(\ref{higham-inn-bound}), $ \sqrt{u \gamma_{2n}(u)} \sqrt{\ln (2 / \lambda)}$ for~(\ref{ilse-inn-bound}), and $\sqrt{ \gamma_{n}(u^2)} \sqrt{1/ \lambda}$ for~(\ref{cheb-inn-bound}). For a constant $\lambda$, we investigate two cases: $nu \ll 1$ and $nu \gg 1$.

For $n$ and $u$ such that $nu \ll 1$ we have $$ \exp{\frac{\sqrt{2n\ln{(2n/\lambda)}}u + nu^2}{1-u}} -1 = u\sqrt{2n\ln{(2n/\lambda)}}  + O(u^2).$$
Moreover,~\cite[Lemma 3.1]{higham2002} implies
$$
\gamma_{n}(u) \leq \frac{nu}{1-nu},
$$
it follows that for $2nu <1$,
\begin{align*}
\sqrt{u \gamma_{2n}(u)} &\leq \sqrt{ \frac{2nu^2}{1-2nu}} = u \sqrt{n} \sqrt{\frac{2}{1-2nu}},
\end{align*}
and for $nu^2 <1$
\begin{align*}
\sqrt{\gamma_{n}(u^2)} &\leq \sqrt{ \frac{nu^2}{1-nu^2}} = u \sqrt{n} \frac{1}{ \sqrt{ 1-nu^2}}.
\end{align*}

Interestingly, for the inner product, at any fixed probability, when $nu \ll 1$,~(\ref{ilse-inn-bound}) and~(\ref{cheb-inn-bound}) bounds are proportional to $\sqrt{n} u$ unlike $\sqrt{n\ln{n}}u$ for the~(\ref{higham-inn-bound}) bound. Note that the deterministic bound is in $O(nu)$.

For $n$ and $u$ such that $nu \gg 1$ and $nu^2 \ll 1$, we have
\begin{align*}
  \exp{\frac{\sqrt{2n\ln{(2n/\lambda)}}u + nu^2}{1-u}} -1 &\approx \exp{\frac{\sqrt{2n\ln{(2n/\lambda)}}u + nu^2}{1-u}}\\
    &\approx \exp{(\sqrt{n\ln{(n)}}u)},
\end{align*}
then
\begin{equation}\label{higham}
\tilde{\gamma}_n(\sqrt{2\ln{(2n/\lambda)}}) \approx \exp{(\sqrt{n\ln{(n)}}u)}.
\end{equation}
Furthermore 
\begin{equation}\label{ipsen}
\sqrt{u\gamma_{2n}(u)} \approx \sqrt{u\exp{(2nu)}-1} \approx \sqrt{u} \exp{(nu)},
\end{equation}
and
\begin{equation}\label{mehdi}
    \sqrt{\gamma_{n}(u^2)} \approx \sqrt{\exp{(nu^2)}-1} \approx \sqrt{n} u + O(u^2).
\end{equation}

Therefore,~(\ref{higham}),~(\ref{ipsen}) and~(\ref{mehdi}) show that~(\ref{cheb-inn-bound}) $\leq $~(\ref{higham-inn-bound}) $\leq $~(\ref{ilse-inn-bound}) when $nu \gg 1$ and $nu^2 \ll 1$.

\subsection{Horner algorithm}
Let us recall all bounds obtained for the Horner algorithm
\begin{align}
\frac{\lvert \widehat{P}(x)  - P(x) \rvert}{\lvert P(x) \rvert} &\leq \mathcal{K}_1 \gamma_{2n}(u),  \quad \label{det-hor-bound}\tag{Det-H}\\
\frac{\lvert \widehat{P}(x)  - P(x) \rvert}{\lvert P(x) \rvert}    &\leq \mathcal{K}_1 \sqrt{u \gamma_{4n}(u)} \sqrt{\ln \frac2\lambda} && \text{with probability $\geq1-\lambda$},  \label{azum-hor-bound}\tag{AH-H}\\ 
\frac{\lvert \widehat{P}(x)  - P(x) \rvert}{\lvert P(x) \rvert}   &\leq \mathcal{K}_1 \sqrt{ \gamma_{2n}(u^2) } \sqrt{\frac1 \lambda} && \text{with probability $\geq1-\lambda$}. \label{cheb-hor-bound}\tag{BC-H}
\end{align}
Similar reasoning to Section~\ref{bound-compare-inner} shows that the probabilistic bounds for the Horner algorithm forward error are in $O(\sqrt{n} u)$ versus $O(nu)$ for the deterministic bound. With the Horner method, the degree of the polynomial, $n$, is seldom very large in practice. 

In conclusion, these probabilistic approaches show that the roundoff error accumulated in $n$ operations is proportional to $\sqrt{n}u$ rather than $nu$. In the next section, we analyze these two probabilistic methods.

\subsection{Bienaymé–Chebyshev vs Azuma-Hoeffding}\label{cheb-vs-azum}
In the following, we compare the three probabilistic bounds~(\ref{higham-inn-bound}), ~(\ref{ilse-inn-bound}) and~(\ref{cheb-inn-bound}) on the inner product forward error (similar reasoning can be applied to the Horner algorithm with the same result). When $nu \ll 1$, at any fixed probability,~(\ref{ilse-inn-bound}) and~(\ref{cheb-inn-bound}) are proportional $O(\sqrt{n} u)$. First, we focus on this case.
The two probabilistic bounds have the same condition number $\mathcal{K}_1$. 
 Thus, it is enough to compare $\sqrt{\frac{u}{2} \gamma_{2n}(u)} \sqrt{2\ln (2 / \lambda)}$ and $\sqrt{ \gamma_{n}(u^2) } \sqrt{1/ \lambda}$. These two bounds depend on $n$ and $\lambda$. Firstly, using the binomial theorem, we have
\begin{align*}
    \frac{u}{2} \gamma_{2n}(u) - \gamma_{n}(u^2) 
    &= \frac{u}{2} \sum_{k=1}^n \dbinom{n}{k} (u^2+2u)^{k} - \sum_{k=1}^n \dbinom{n}{k} (u^2)^{k}\\
    &\geq \sum_{k=1}^2 \dbinom{n}{k} \left[\frac{u}{2}(u^2+2u)^k -(u^2)^k\right]
    \geq n(n-\frac12)u^3.
\end{align*}
We can conclude that
\begin{equation}\label{inequality}
    \sqrt{ \gamma_{n}(u^2) } \leq \sqrt{\frac{u}{2} \gamma_{2n}(u)} \ \text{for all} \ n\geq 1.
\end{equation}
Now, let us compare $\sqrt{1/\lambda}$ and $\sqrt{2\ln(2/\lambda)}$ for $\lambda \in]0;1[$, 
\begin{figure}
\centering
\begin{minipage}[t]{.48\textwidth}
  \centering
  \includegraphics[width=.99\linewidth]{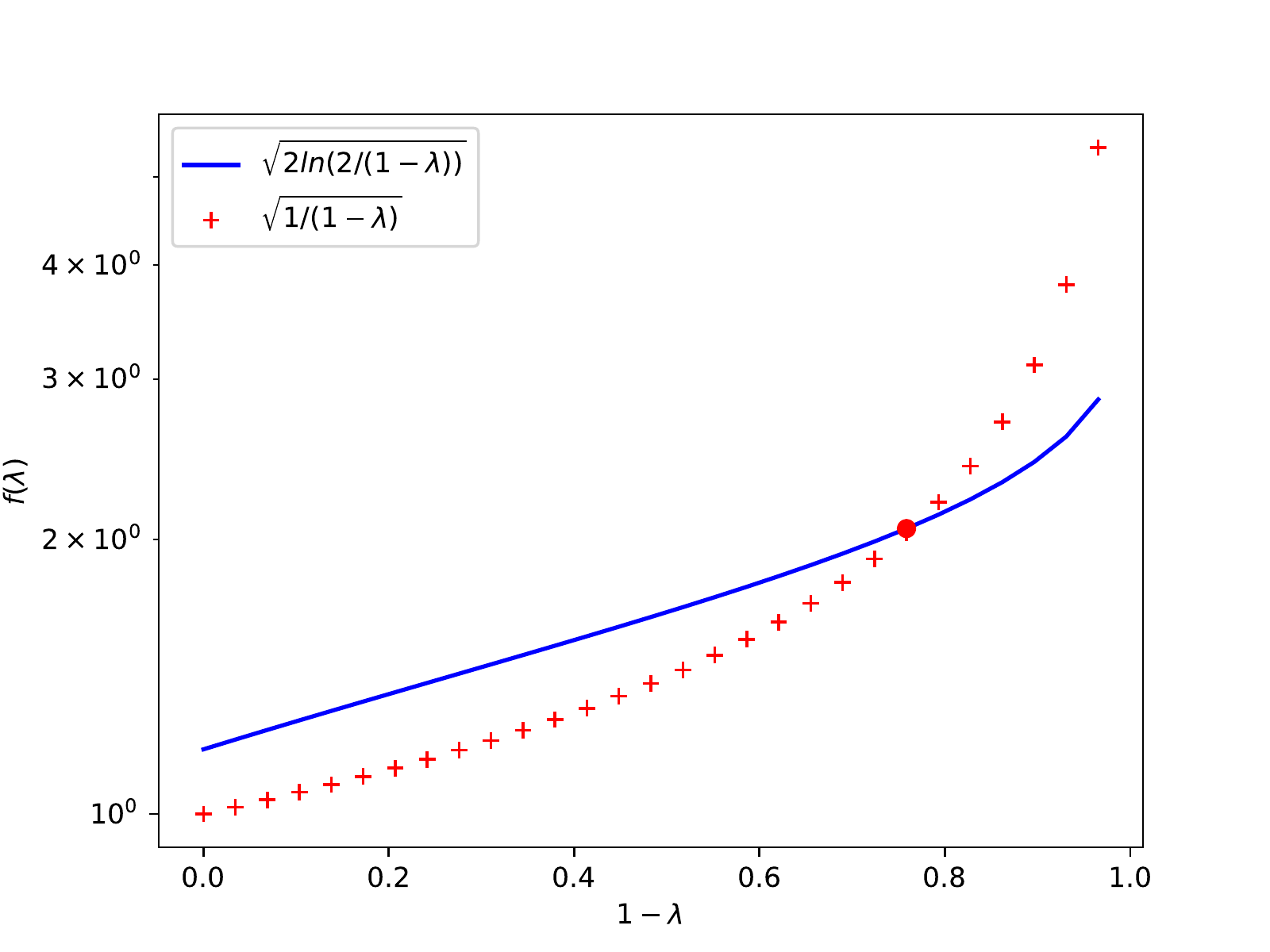}
    \caption{Illustration of $\sqrt{1/\lambda}$ and $\sqrt{2\ln(2/\lambda)}$ behaviour for all $\lambda \in]0;1[$.}
    \label{compare lambda}
\end{minipage}%
\hfill
\begin{minipage}[t]{.48\textwidth}
  \centering
  \includegraphics[width=.99\linewidth]{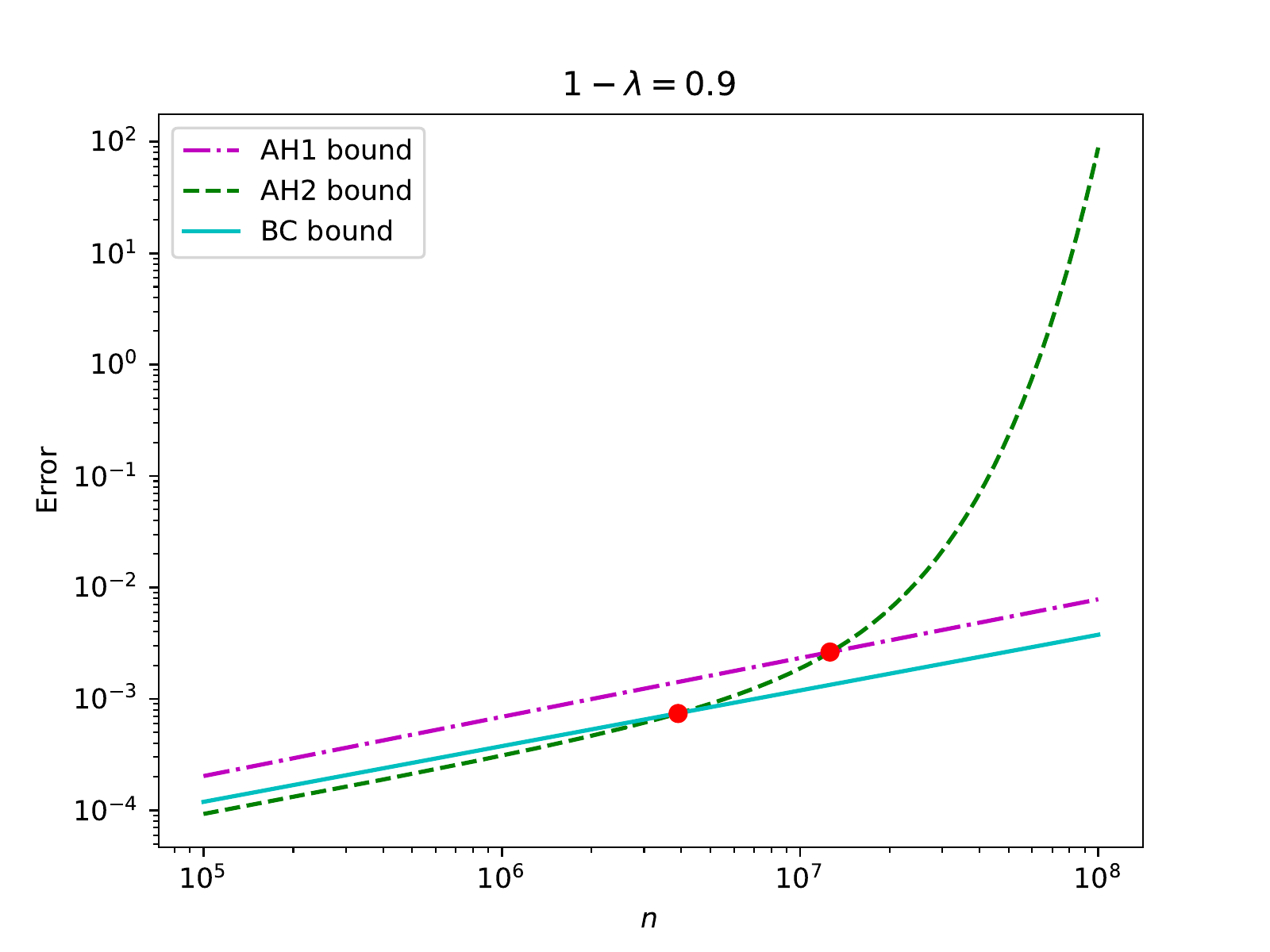}
    \caption{AH1, AH2 and BC bounds with probability $0.9$ and $u=2^{-23}$ for the inner product.}
    \label{fig:bc-ah1-ah2}
\end{minipage}
\end{figure}

Figure~\ref{compare lambda} and the inequality~(\ref{inequality}) show that whatever the problem size $n$ and for a probability at most $\approx 0.758$, the BC method gives a tighter probabilistic bound than the AH2 method.

Figure~\ref{fig:bc-ah1-ah2} confirms the discussion of Section~\ref{bound-compare-inner} and it shows that with a probability $0.9$, when $nu \gg 1$, AH2 bound grows rapidly compared to AH1 and BC bounds. Regarding BC bound, the variance calculation and the mean independence allow to bound the error terms $(1+\delta)^2$ by $(1+u^2)$ and avoid all $\delta$ terms of degree one because $E(\delta)=0$. In contrast, the AH1 and AH2 methods require bounded increments leading to terms $(1+u)^2$. As $n$ increases, the advantage of Azuma-Hoeffding inequality for a probability near $1$ becomes negligible. 

For all asymptotic comparisons between the bounds in this paper, we have chosen to work with $u\rightarrow0$, $n\rightarrow\infty$ and fixed probability $\lambda$, which we think adapted to many if not most current practical use cases. A situation with $\lambda\rightarrow0$ and fixed $n$ gives the advantage to the Azuma-Hoeffding bounds over the Bienaymé-Chebyshev one.

\begin{table}
   {\renewcommand{\arraystretch}{1.1}

  \centering

    \begin{tabular}{|C{2.5cm}|C{2.5cm}|C{3.5cm}|C{2.5cm}|}
\hline Probability   & $u$ & Precision format & $n \gtrsim $  \\

\hline $1- \lambda = 0.95$  & $2^{-7}$ & bfloat16 & $110$  \\
\cline{2-4} 
& $2^{-10}$ & binary16 & $890$  \\
\cline{2-4} 
& $2^{-23}$ & binary32 & $7.3~\text{e}06$  \\
\cline{2-4} 
& $2^{-52}$ & binary64 & $3.9~\text{e}15$\\
\hline $1- \lambda = 0.99$  & $2^{-7}$ & bfloat16 & $220$  \\
\cline{2-4}
& $2^{-10}$ & binary16 & $1810$\\
\cline{2-4} 
& $2^{-23}$ & binary32 & $1.48~\text{e}07$\\
\cline{2-4} 
& $2^{-52}$ & binary64 & $7.9~\text{e}15$\\
\hline
\end{tabular}                                 
\caption{The smallest $n$ such that BC method gives a tighter probabilistic bound than AH2 method for the inner product.}
    \label{table-n}
    }
\end{table}

Table~\ref{table-n} illustrates how BC is tighter than AH2 when $n$ grows. The $n$ threshold above which BC is preferable to AH2 bound depends on the format precision. The lower the precision, the lower the threshold becomes. Using SR in low precision is of high interest in the areas of machine learning~\cite{gupta}, PDEs~\cite{pde}, and ODEs~\cite{ode}, motivating the use of our improved BC method. 

\section{Numerical experiments}\label{sec:exp}
This section presents numerical experiments that support and complete the theory presented previously. The various bounds are compared on two numerical applications: the inner product and the evaluation of the Chebyshev polynomial. 

We show that the probabilistic bounds are tighter than the deterministic bound and faithfully capture the behavior of SR-nearness forward error. For an inner product of large vectors, we show that BC bound is smaller than AH1 and AH2 bounds. All SR computations are repeated $30$ times with verificarlo~\cite{verificarlo}; we plot all samples and the forward error of the average of the 30 SR instances.

\subsection{Horner algorithm} 
We use Horner’s method to evaluate the polynomial $P(x)=T_{N}(x) = \sum_{i=0}^{\llfloor \frac{N}{2} \rrfloor} a_i (x^2)^i$ where $T_{N}$ is the Chebyshev polynomial of even degree $N=2n$. The previous error bounds, ~(\ref{det-hor-bound}),~(\ref{azum-hor-bound}), and~(\ref{cheb-hor-bound}) apply to this computation.

\begin{figure}
    \centering
    \subfloat{
         \hspace{-.6cm}\includegraphics[scale=0.41]{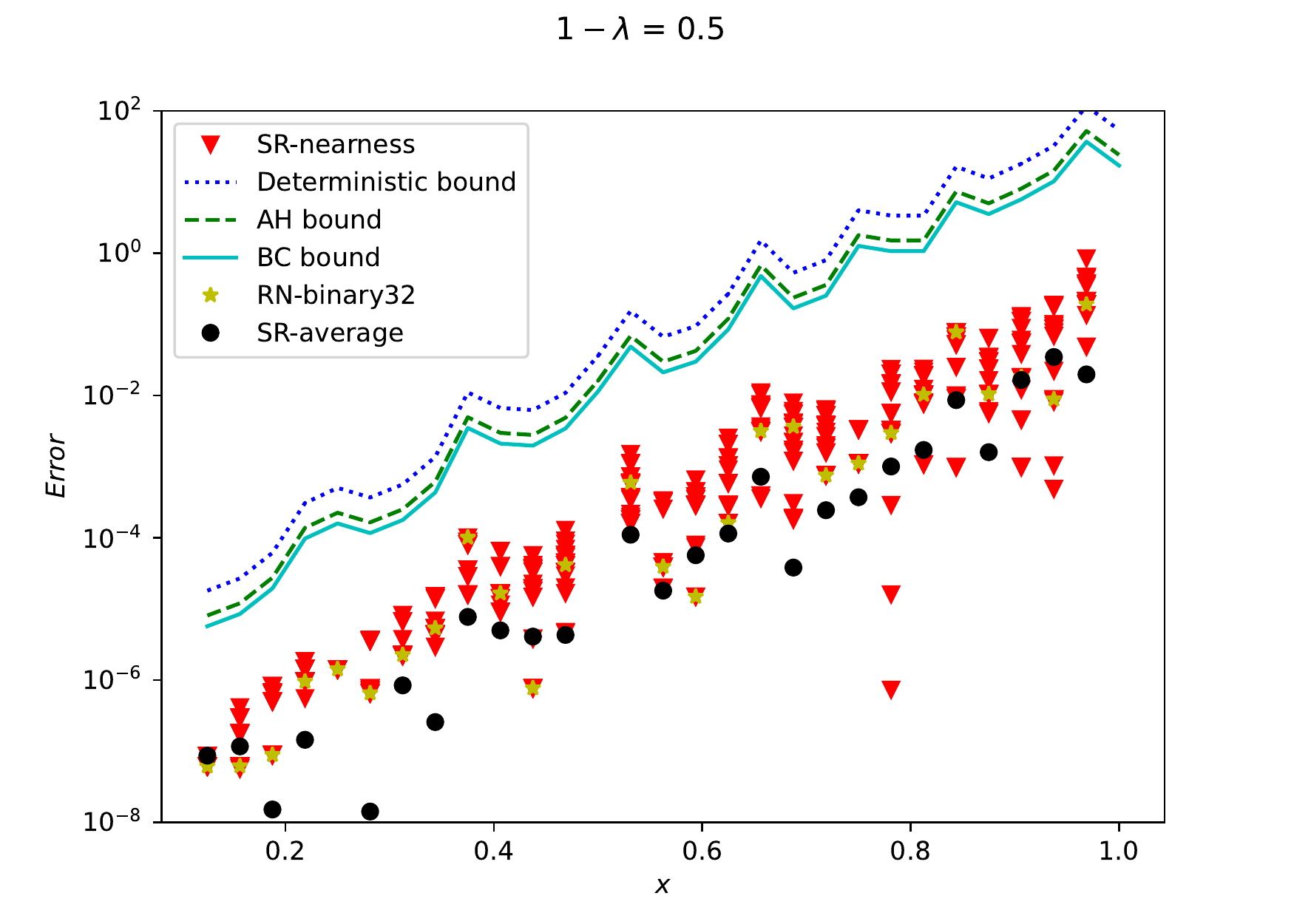}
} \subfloat{
        \hspace{-0.8cm}\includegraphics[scale=0.41]{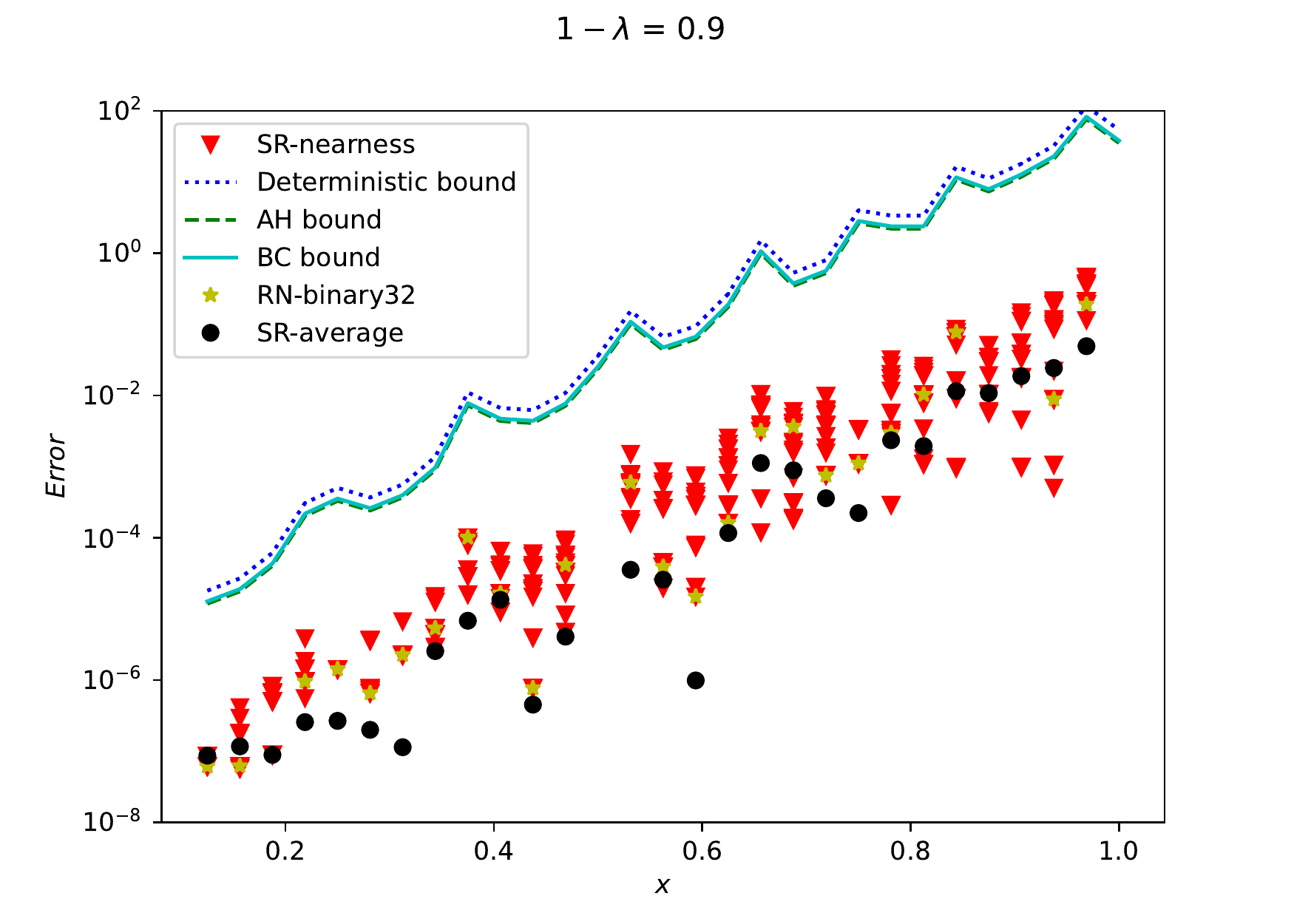}\hspace{-0.8cm}
}
\caption{Probabilistic error bounds with probability $1- \lambda =0.5$ (left) and $1- \lambda =0.9$ (right) vs deterministic bound for the Horner's evaluation of $T_{20}(x)$ and $u=2^{-23}$.
Triangles mark $30$ instances of the SR-nearness relative errors evaluation in binary32 precision, a circle marks the relative errors of the $30$ instances average, and a star represents the IEEE RN-binary32 value.  }
\label{fig:n=20}
\end{figure}

Chebyshev polynomial is ill-conditioned near $1$ as shown in Figure~\ref{fig:n=20}, which evaluates $T_{20}(x)$ for $x\in[\frac{8}{64};1]$. Due to catastrophic cancellations among the polynomial terms, the condition number increases from $10^0$ to $10^7$ in the chosen $x$ interval, resulting in an increasing numerical error for both RN-binary32 and SR-nearness computations.

The left plot confirms that the Bienaymé–Chebyshev bound~(\ref{cheb-hor-bound}) is more accurate than the Azuma-Hoeffding bound~(\ref{azum-hor-bound}) for probability $1 - \lambda = 0.5$. With a higher probability $1 - \lambda = 0.9$ (right plot), since $N=20$ and $u=2^{-23}$ Azuma-Hoeffding bound~(\ref{azum-hor-bound}) is tighter, as predicted in Figure~\ref{fig:bc-ah1-ah2}. Both probabilistic bounds are tighter than the deterministic bound. 
For $N=20$, there is no significant difference between SR-nearness and RN-binary32. However, as expected, the average of the SR-nearness computations is more precise than the nearest round evaluation for almost all values of $x$.

\begin{figure}
    \centering
    \subfloat{
         \hspace{-.6cm}\includegraphics[scale=0.44]{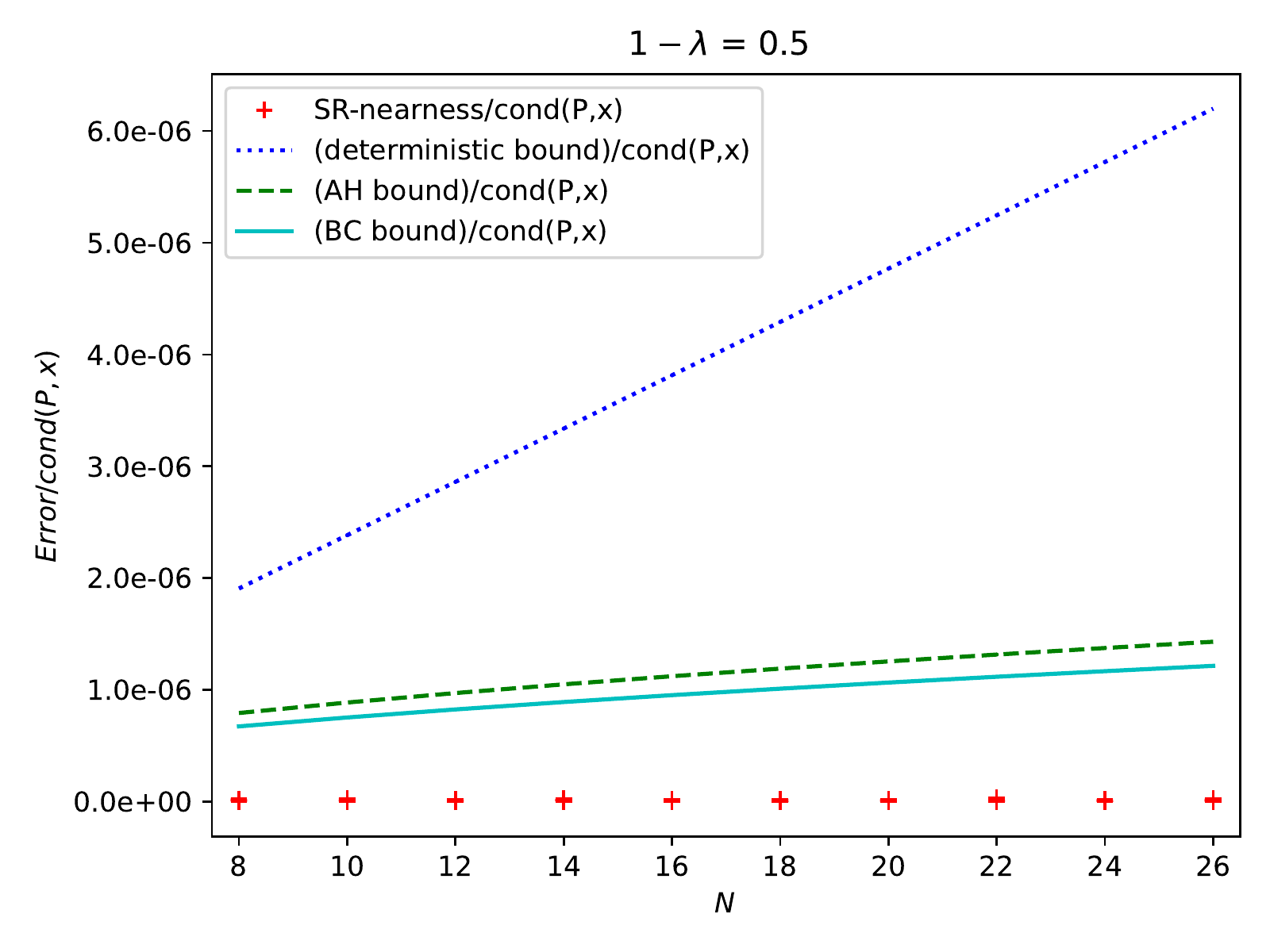}
} \subfloat{
        \hspace{-0.3cm}\includegraphics[scale=0.44]{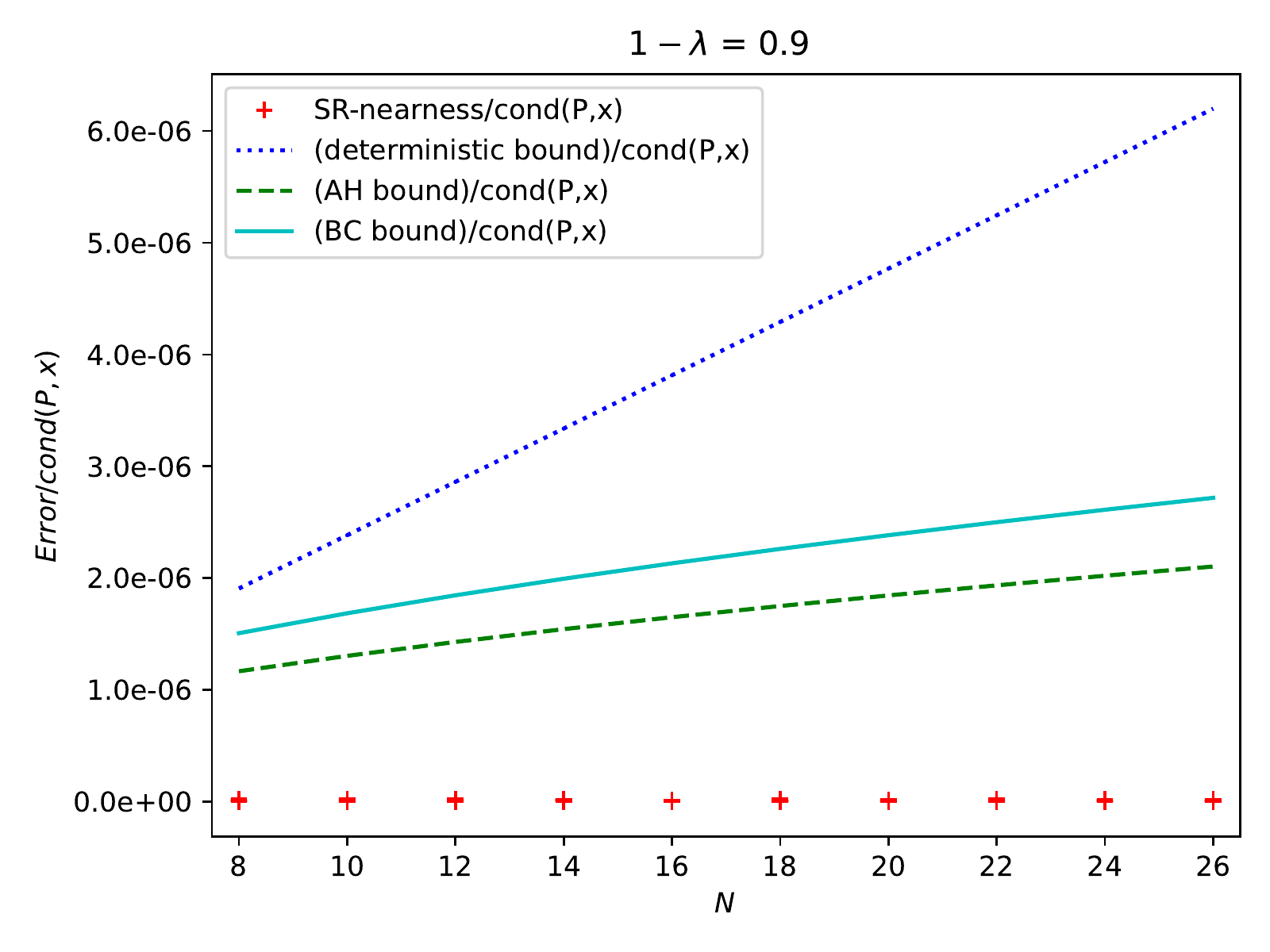}\hspace{-0.9cm}
}
\caption{Normalized forward error ($\text{error}/cond(P,x)$) with probability $1- \lambda =0.5$ (left) and $1- \lambda =0.9$ (right) for Horner's evaluation of $T_{N}(24/26)$. 
}
    \label{fig:x=24/26}
\end{figure}

In Figure \ref{fig:x=24/26}, the three previous bounds and the forward error are normalized by the condition number $cond(P,x)$. The evaluation in  $x=24/26\approx0.923$ is plotted for various polynomial degrees $N.$ As expected, when $N$ increases, the deterministic bound grows faster than the probabilistic bounds. The right plot shows that Azuma-Hoeffding bound is tighter for a high probability and a small $n$. 
Overall, Chebyshev polynomial numerical experiment illustrates the advantage of the probabilistic error bounds over the deterministic error bound.
However, for most of the evaluations in this experiment, RN-binary32 is more accurate than one instance of SR-nearness. This result is unsurprising because the degree $n$ is small. To illustrate the behavior of these errors with a large $n$, we now turn to the inner product.

\subsection{Inner product}
To showcase the advantage of using BC method for large $n$, we present a numerical application of the inner product for vectors with positive floating-points chosen uniformly at random between $0$ and $1$. 

\begin{figure}
    \centering
   \includegraphics[scale=0.4]{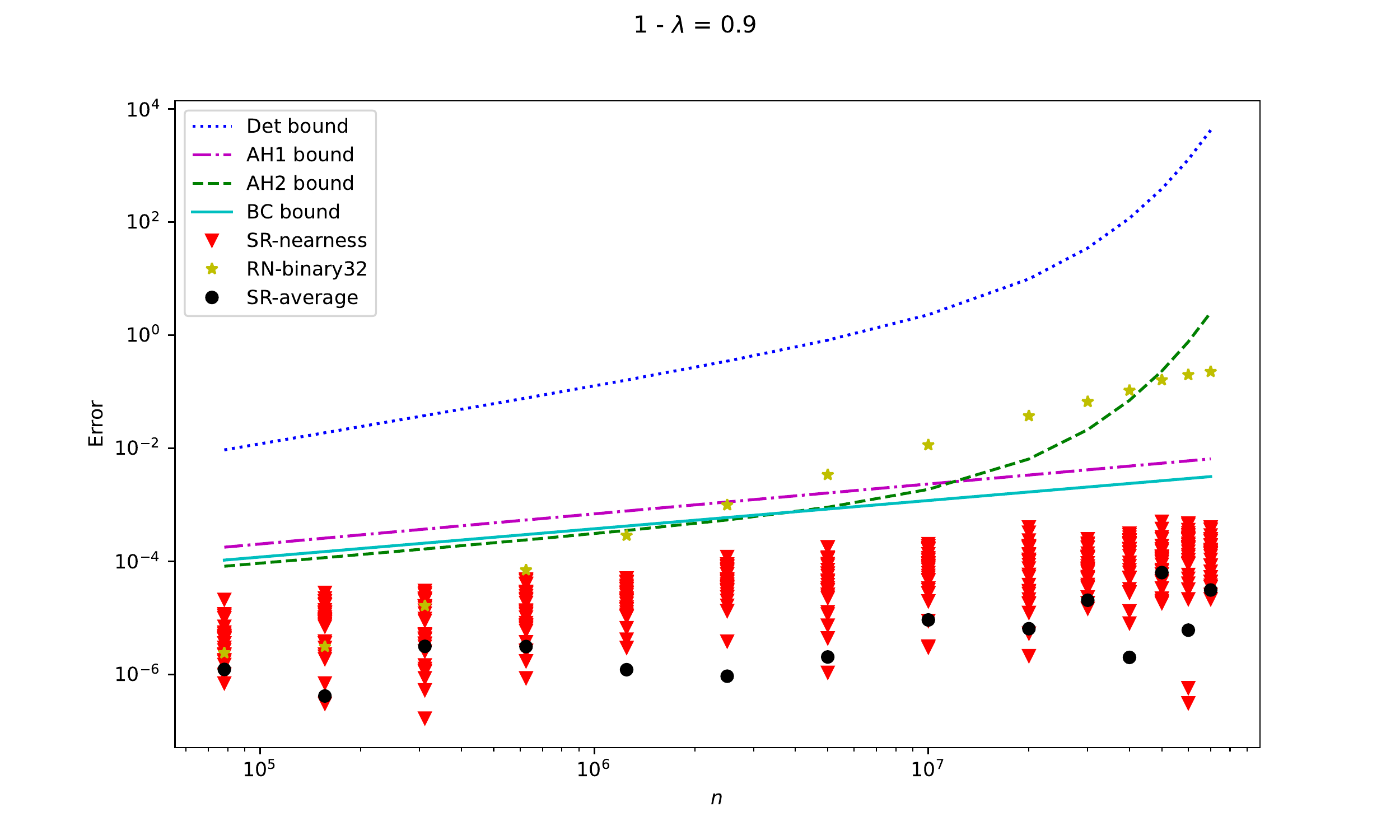}
    \caption{Probabilistic bounds with probability $1- \lambda =0.9$ vs deterministic bound of the computed forward errors of the inner product of with $u=2^{-23}$.}
    \label{fig:inner}
\end{figure}

For small $n$, AH1, AH2, and BC bounds are comparable with a slight advantage for~(\ref{ilse-inn-bound}).
However, as shown in Section~\ref{bound-compare-inner}, when $nu \gg 1$, the AH2 bound grows exponentially faster than AH1 and BC bounds. Asymptotically, the AH1 and BC bounds are therefore much tighter.

Interestingly, when $n$ increases, a single instance of SR-nearness in binary32 precision is more accurate than RN-binary32. This is because the summation terms are chosen uniformly between 0 and 1. The terms closest to zero are absorbed. With RN-binary32 the absorption errors are biased and will add up, while SR avoids stagnation and mitigates absorption errors.
If we choose the terms in $[-1; 1]$, SR and RN-binary32 have the same behavior. In this case, the absorption errors for RN-binary32 compensate because positive and negative errors are  uniformly distributed. If we choose the terms in $[1/2; 1]$, no absorption occurs for $n < 2^{23}$, and on this domain, SR and RN-binary32 behave similarly.

\begin{figure}
    \centering
   \includegraphics[scale=0.36]{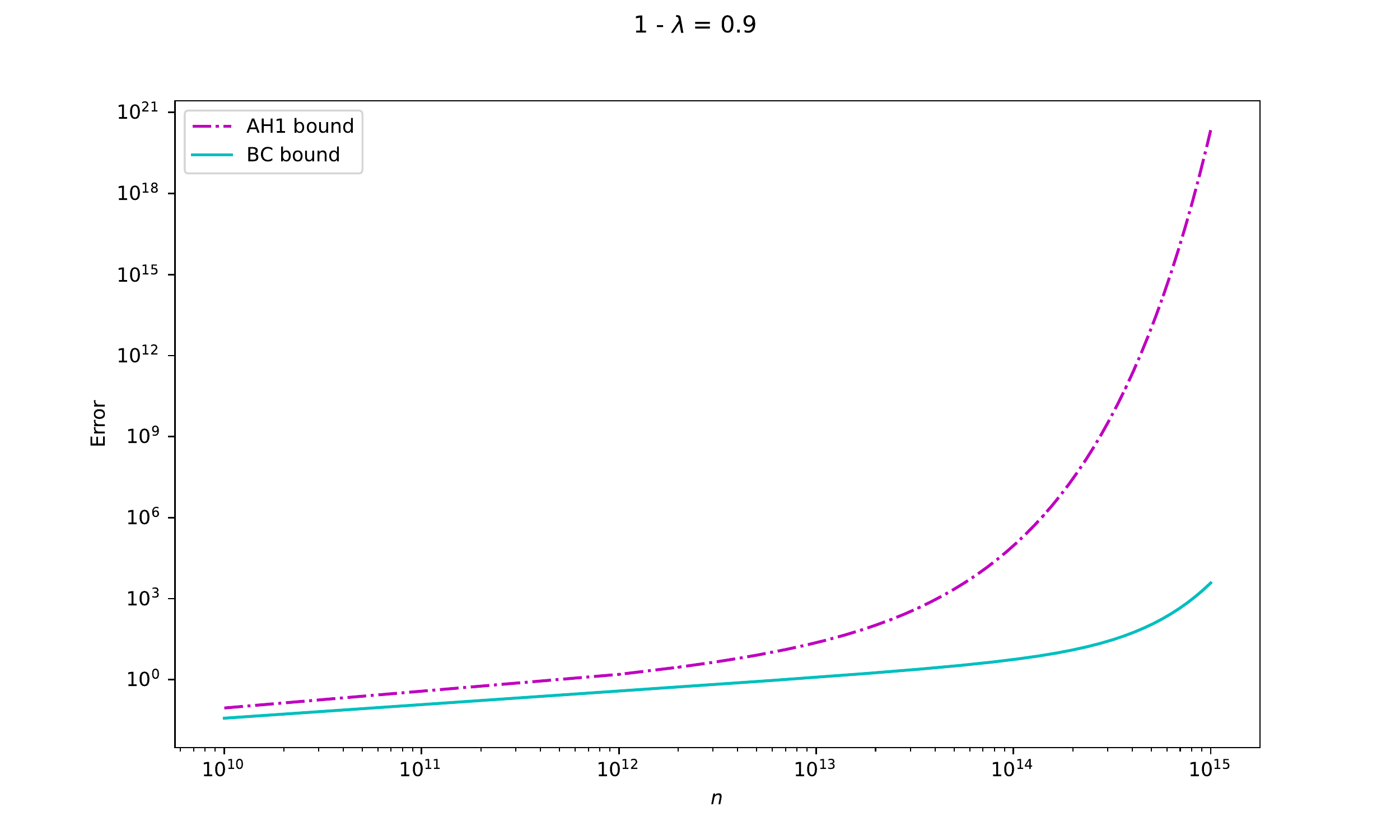}
    \caption{AH1 bound vs BC bound with probability $1- \lambda =0.9$ and $u=2^{-23}$ for the inner product of with.}
    \label{fig:bc-ah1}
\end{figure}

Figure~\ref{fig:bc-ah1} illustrates the advantage of using~(\ref{cheb-inn-bound}) and shows that for a large $n \geq  10^{13}$ and $u=2^{-23}$, the AH1 bound increases faster than the BC bound.

For large vectors, using stochastic rounding instead of the default round to nearest improves the computation accuracy of the inner product. 
However, this experiment raises concerns regarding the use of SR as a model to estimate RN rounding errors~\cite{verificarlo,sohier2021confidence}, in particular for a large number of operations. Further studies are required to assess precisely the limits of this model and possibly give criteria to detect them. 

\section{Conclusion}

For a wide field of applications, SR results in a smaller accumulated error, for example by avoiding stagnation effects. Moreover, SR errors satisfy the mean independence property allowing to derive tight probabilistic error bounds from either our variance bound or the martingale property.

For an inner product $y=a^{\top}b,$ Sub-section~\ref{azuma-method} compares the benefits of constructing the martingale from the recursive summation of the inner product~\cite{ilse} versus the construction from the errors accumulated in the whole process at each product $a_ib_i$~\cite{theo21stocha}. In particular, with a fixed probability, the construction in~\cite{ilse} gives a $O(\sqrt{n}u)$ probabilistic bound, tighter than the $O\big(u\sqrt{n\ln{(n)}} \big)$ bound in~\cite{theo21stocha} when $nu \ll 1$. Nevertheless, when $nu \gg 1$, Figure~\ref{fig:inner} shows that the~(\ref{ilse-inn-bound}) bound increases faster than~(\ref{higham-inn-bound}) bound.

An extension of the method in~\cite{ilse} to the Horner algorithm is presented.
Unlike the inner product, Horner algorithm does not explicitly satisfy the martingale property on the partial sums requiring a change of variable before one can use the Azuma-Hoeffding inequality.

Lemma~\ref{model} is a variance bound for the family of algorithms whose error can be written as a product of error terms of the form $1+ \delta$. Based on the Bienaymé–Chebyshev inequality, a new method is proposed to obtain probabilistic error bounds. This method allows to get tighter probabilistic error bound in various situations, such as computations with a large $n$. We demonstrate the strength of this new approach on two algorithms: the inner product which has been previously studied, and Horner polynomial evaluation, for which no SR results were known beforehand.

The scripts for reproducing the numerical experiments in this paper are published in the repository \url{https://github.com/verificarlo/sr-variance-bounds/}.

\section*{Acknowledgment}
We thank the anonymous reviewers for their insightful comments that have substantially improved the paper.

\bibliographystyle{siamplain}
\bibliography{references}

\end{document}